\documentclass[draft,12pt]{amsart}

\usepackage{amssymb,amscd}
\usepackage[all]{xy}

\usepackage{latexsym}

\usepackage{amsfonts}
\usepackage{amssymb}
\usepackage{mathrsfs}

\usepackage{graphicx}

\newcommand{\bneq}{\mathbin{\rotatebox[origin=c]{90}{$\neq$}}}

\DeclareMathOperator{\bv}{\boldsymbol{v}}
\DeclareMathOperator{\bw}{\boldsymbol{w}}
\newcommand{\norm}[1]{\lVert#1\rVert}

\theoremstyle{plain}
\newtheorem{theorem}{Theorem}[section]
\newtheorem{corollary}[theorem]{Corollary}
\newtheorem{lemma}[theorem]{Lemma}
\newtheorem{proposition}[theorem]{Proposition}

\theoremstyle{plain}

\def\R{{\rm I\kern-2ptR}}
\def\N{{\rm I\kern-2ptN}}

\def\qed{\hskip .6em \raise1.8pt\hbox{\vrule height4pt
width6pt depth2pt}}
\def\qedd{\hskip .4em \raise1.8pt\hbox{\vrule height3pt
width5pt depth1.8pt}}
\def\sq{\hskip .6em \raise1.8pt\hbox{\vrule
height4pt width6pt depth2pt}}

\def\leaderfill{\leaders\hbox to 1em{\hss.\hss}\hfill}

\begin{document}

\title[]{Closed ideals in $\mathcal{L}(X)$ and $\mathcal{L}(X^*)$ when $X$ contains certain copies of $\ell_p$ and $c_0$}

%\author{G. Sirotkin}
%\address{Department of Mathematical Sciences, Northern Illinois University, DeKalb, IL 60115}
%\email{sirotkin@math.niu.edu}

%\author{V. G. Troitsky}
%\address{Department of Mathematical and Statistical Sciences, University of Alberta, Edmonton, AB, Canada T6G 2R3}
%\email{troitsky@ualberta.ca}

\author{Ben Wallis}
\address{Department of Mathematical Sciences, Northern Illinois University, DeKalb, IL 60115}
\email{wallis@math.niu.edu}

\begin{abstract}Suppose $X$ is a real or complexified Banach space containing a complemented copy of $\ell_p$, $p\in(1,2)$, and a copy (not necessarily complemented) of either $\ell_q$, $q\in(p,\infty)$, or $c_0$.  Then $\mathcal{L}(X)$ and $\mathcal{L}(X^*)$ each admit continuum many closed ideals.  If in addition $q\geq p'$, $\frac{1}{p}+\frac{1}{p'}=1$, then the closed ideals of $\mathcal{L}(X)$ and $\mathcal{L}(X^*)$ each fail to be linearly ordered.  We obtain additional results in the special cases of $\mathcal{L}(\ell_1\oplus\ell_q)$ and $\mathcal{L}(\ell_p\oplus c_0)$, $1<p<2<q<\infty$.\end{abstract}

\maketitle
\theoremstyle{plain}

\section{Introduction}

The past decade has seen some dramatic new results on the closed ideal structure of the algebra of operators $\mathcal{L}(\ell_p\oplus\ell_q)$, $1<p<q<\infty$.  Some important new ideals in that algebra were described in \cite{SSTT07} and \cite{Sc12} when $1<p<2<q<\infty$, and then in \cite{SZ14} the authors showed that it contains infinitely many closed ideals for all choices $1<p<q<\infty$.  In this paper, we find that only small changes to the proofs are necessary to adapt one of the main results in \cite{SZ14}, yielding infinitely many---indeed, continuum many---new closed ideals in $\mathcal{L}(\ell_p\oplus c_0)$, for $1<p<2$, and in $\mathcal{L}(\ell_1\oplus\ell_q)$ for $2<q<\infty$.  We then adapt results from \cite{SSTT07} to find additional information on the closed ideal structure of these operator algebras.

In the process of doing all this, we noticed that the proof methods remain valid for much more general cases, yielding the two main Theorems \ref{main1} and \ref{main2} below.  Before stating these, let us recall some definitions and notation.  If $X$ is a real Banach space then $X_\mathbb{C}$ denotes its complexification.  Recall that an operator $T\in\mathcal{L}(X,Y)$, $X$ and $Y$ Banach spaces, is said to be {\bf finitely strictly singular} ($\mathcal{FSS}$) just in case for every $\epsilon>0$ there exists $n\in\mathbb{Z}^+$ such that for every $n$-dimensional subspace $E\subseteq X$ we have $\inf_{x\in E}\norm{Tx}<\epsilon\norm{x}$.  (In the literature, finitely strictly singular operators are sometimes called {\it superstrictly singular} operators.)  Milman proved in \cite{Mi69} that class $\mathcal{FSS}$ forms a norm-closed operator ideal.  Let us also define, for an operator $T\in\mathcal{L}(X,Y)$ and each $n\in\mathbb{Z}^+$,
\[a_n(T)=\sup\inf\left\{\norm{Tx+E}_{Y/E}:x\in X,\norm{x}=1\right\},\]
where the ``sup'' is taken over all closed subspaces $E$ of $Y$ such that $\dim(Y/E)=n$.  The operator $T$ is then said to be {\bf superstrictly cosingular} ($\mathcal{SSCS}$) just in case $\lim_{n\to\infty}a_n(T)=0$.  It can be shown (cf., e.g., \cite[Theorem 4]{Pl04}) that $\mathcal{SSCS}$ is in full duality with $\mathcal{FSS}$, and hence forms a norm-closed operator ideal.  More precisely, an operator $T$ is class $\mathcal{SSCS}$ (resp. $\mathcal{FSS}$) if and only if $T^*$ is class $\mathcal{FSS}$ (resp. $\mathcal{SSCS}$).

Our first main result, then, is as follows.

\begin{theorem}\label{main1}Let $p\in(1,2)$ and $q\in(p,\infty)$.  Suppose $X$ is a real Banach space containing a complemented copy of $\ell_p$, and a copy of either $\ell_q$ or $c_0$ (which need not be complemented).  Then $\mathcal{L}(X)$ and $\mathcal{L}(X_\mathbb{C})$ each admit a chain, with cardinality of the continuum, of closed ideals contained in $\mathcal{FSS}(X)$ and $\mathcal{FSS}(X_\mathbb{C})$, respectively.  Furthermore, $\mathcal{L}(X^*)$ and $\mathcal{L}(X_\mathbb{C}^*)$ each admit a chain, with cardinality of the continuum, of closed ideals contained in $\mathcal{SSCS}(X^*)$ and $\mathcal{SSCS}(X_\mathbb{C}^*)$, respectively.\end{theorem}

Among the spaces satisfying the conditions of Theorem \ref{main1} are Rosenthal's $X_p$ spaces for $p\in(1,2)\cup(2,\infty)$, defined in \cite{Ro70}.  Indeed, let us consider the generalization of those spaces defined in \cite{Woo75}, denoted $X_{p,r}$ for $1\leq r<p\leq\infty$.  Note that whenever $p\in(2,\infty)$, we have $X_p=X_{p,2}$, and whenever $p\in(1,2)$ we have $X_p=X_{p',2}^*$.  In \cite[Corollary 3.2]{Woo75} it was proved that $X_{p,r}$ always contains complemented copies of $\ell_p$ (or $c_0$, if $p=\infty$) and $\ell_r$.  They are reflexive for all $1<r<p<\infty$, and hence satisfy the conditions of Theorem \ref{main1} in those cases.  The nonreflexive space $X_{\infty,r}$ also satisfies the conditions as long as $r\in(1,2)$.  So do certain Orlicz sequence spaces, for instance the ones described in \cite[Corollary 4.9]{Li73}, which contain complemented subspaces of $\ell_p$ for all $p\in[a,b]$, where $a\in[1,2)$ and $b\in(a,\infty)$.  However, to be sure, not all Orlicz sequence spaces satisfy the conditions, since there exist such spaces, for instance the ones constructed in \cite{LT72}, with no complemented copies of $c_0$ or $\ell_p$ for any $p\in[1,\infty)$.  Most notably, the spaces $\ell_p\oplus c_0$ and $\ell_1\oplus\ell_q$ satisfy the conditions of Theorem \ref{main1} for all $1<p<2<q<\infty$, and we shall discuss them at greater length in section 5.

We obtain the following additional result on the structure of closed ideals under certain similar conditions.

\begin{theorem}\label{main2}Let $1<p<2<p'\leq q<\infty$.  Suppose $X$ is a (real or complex) Banach space containing a complemented copy of $\ell_p$, and also containing a copy of either $\ell_q$ or $c_0$ (not necessarily complemented).  Then each of $\mathcal{L}(X)$ and $\mathcal{L}(X^*)$ contains two incomparable closed ideals.\end{theorem}

\noindent Note that, by two ideals being {\it incomparable}, we mean that neither ideal is a subset of the other.  This means in particular that the closed ideals of $\mathcal{L}(X)$ and $\mathcal{L}(X^*)$ in the above Theorem are not linearly ordered.

Let us set forth some notation which shall be used throughout.  For the most part, our notation will be standard, such as appears in \cite{LT77,AA02,AK06}.  However, we shall recall presently some of the most common conventions.  If $1\leq p\leq\infty$ then let $p'\in[1,\infty]$ denote its conjugate, i.e. $\frac{1}{p}+\frac{1}{p'}=1$.   For any set $S$, denote by $|S|$ its cardinality.  For normed spaces $X$ and $Y$, we write $\mathcal{L}(X,Y)$ for the space of all continuous linear operators from $X$ into $Y$.  Indeed, by an {\it operator} we shall always mean a continuous linear operator between normed linear spaces.  We let $\mathcal{K}$ denote the class of all compact operators.  If $A$ is a subset of a Banach space $X$ then we denote by $\overline{A}$ its closure and $[A]$ its closed linear span.  Let us also borrow a piece of terminology from \cite{Sc12}:  If $X$ and $Y$ are Banach spaces, then a linear subspace $\mathcal{J}$ of $\mathcal{L}(X,Y)$ is called a {\bf subideal} just in case whenever $A\in\mathcal{L}(X)$, $B\in\mathcal{L}(Y)$, and $T\in\mathcal{J}$, we have $BTA\in\mathcal{J}$.  (A subideal of $\mathcal{L}(X)$ is called, simply, an {\it ideal}.)  Whenever $\mathcal{A}$ is a set of continuous linear operators, we let
\[\mathcal{G}_\mathcal{A}(X,Y)=\{T\in\mathcal{L}(X,Y):\text{there exists }A\in\mathcal{A}\text{ such that }T\text{ factors through }A\}.\]
In case $\mathcal{A}$ is just a singleton $\{R\}$, we write $\mathcal{G}_R=\mathcal{G}_{\{R\}}$.  For a class $\mathcal{J}$ of continuous linear operators between Banach spaces, we write $[\mathcal{J}]$ for the class whose components are just the closed linear spans of the components of $\mathcal{J}$.  In other words, $[\mathcal{J}](X,Y)=[\mathcal{J}(X,Y)]$.

If $X$ and $Y$ are Banach spaces with respective bases $(x_n)$ and $(y_n)$ such that $(x_n)$ {\it dominates} $(y_n)$, i.e. $\norm{\sum a_ny_n}\leq C\norm{\sum a_nx_n}$ for all $(a_n)\in c_{00}$, then there exists a natural map $I_{X,Y}\in\mathcal{L}(X,Y)$ which we shall call the {\bf formal identity operator}, i.e. the operator satisfying $I_{X,Y}x_n=y_n$ for all $n\in\mathbb{N}$.  Let $I_{p,q}$ denote the formal identity from $\ell_p$ to $\ell_q$, $1\leq p\leq q<\infty$, and let $I_{p,0}$ denote the formal identity from $\ell_p$ to $c_0$ (all with respect to the canonical bases).  We also write $I_{p,\infty}$ and $I_{0,\infty}$ for the operators taking the respective canonical bases of $\ell_p$ and $c_0$ into $\ell_\infty$ in the obvious way.  Milman observed in \cite{Mi70} that $I_{p,q}\in\mathcal{FSS}(\ell_p,\ell_q)$ and $I_{p,0}\in\mathcal{FSS}(\ell_p,c_0)$ for all $1\leq p<q\leq\infty$.

If $1\leq p<\infty$, let us write
\[Z_p:=\left(\bigoplus_{n=1}^\infty\ell_2^n\right)_{\ell_p},\]
and
\[Z_\infty :=\left(\bigoplus_{n=1}^\infty\ell_2^n\right)_{c_0}.\]
The spaces $Z_p$, $1\leq p\leq\infty$, each have a natural basis which is formed by stringing together the standard bases of $\ell_2^n$.  Let us call this natural basis the {\bf canonical basis} for $Z_p$.  Pe\l czy\'{n}ski proved (cf., e.g., \cite[p73]{LT77}) that for any $1<p<\infty$ there exists an isomorphism
\[D_p:\ell_p\to Z_p=\left(\bigoplus_{n=1}^\infty\ell_2^n\right)_{\ell_p}.\]
(Note that although this means $Z_p\cong\ell_p$, the canonical bases of these spaces are quite different.)  If $1\leq p\leq q<\infty$, we then denote by
\[I_{2,p,q}:Z_p=\left(\bigoplus_{n=1}^\infty\ell_2^n\right)_{\ell_p}\to\left(\bigoplus_{n=1}^\infty\ell_2^n\right)_{\ell_q}=Z_q\]
and
\[I_{2,p,0}:Z_p=\left(\bigoplus_{n=1}^\infty\ell_2^n\right)_{\ell_p}\to\left(\bigoplus_{n=1}^\infty\ell_2^n\right)_{c_0}=Z_\infty,\]
the formal identity operators between corresponding canonical basis vectors.

Unfortunately, no such Pe\l czy\'{n}ski decomposition exists for $\ell_1$ or $c_0$.  However, we can nevertheless construct a continuous linear embedding as follows.  Recall that for every $n\in\mathbb{Z}^+$ there exists $k_n\in\mathbb{Z}^+$ and a 2-embedding $\theta_n:\ell_2^n\to\ell_\infty^{k_n}$ (cf., e.g., \cite[Example 11.1.2]{AK06}).  Then we write
\[\theta=\bigoplus_{n=1}^\infty\theta_n:Z_\infty =\left(\bigoplus_{n=1}^\infty\ell_2^n\right)_{c_0}\to\left(\bigoplus_{n=1}^\infty\ell_\infty^{k_n}\right)_{c_0}=c_0.\]
Let us fix $D_p$ ($1<p<\infty$), $(k_n)_{n=1}^\infty$, $(\theta_n)_{n=1}^\infty$, and $\theta=\bigoplus_{n=1}^\infty\theta_n$ once and for all.  Then we can define a {\bf $\boldsymbol{(p,q)}$-Pelczyski decomposition operator} as any operator of the form $D_q^{-1}I_{2,p,q}D_p\in\mathcal{L}(\ell_p,\ell_q)$, $1<p<q<\infty$, and a {\bf $\boldsymbol{(p,0)}$-left Pe\l czy\'{n}ski decomposition operator} as any operator of the form $\theta I_{2,p,0}D_p\in\mathcal{L}(\ell_p,c_0)$, $1<p<\infty$.

We organize the remainder of this paper as follows.  Section 2 lays out some preliminaries on how closed subideals in real Banach spaces are related to closed subideals in their complexified counterparts.  This is necessary for section 3, where the main result of \cite{SZ14} is adapted to prove the more general Theorem \ref{main1}.  In section 4, we modify a result from \cite{SSTT07} to prove Theorem \ref{main2}.  Finally, in section 5, we prove some additional results for the special cases $\mathcal{L}(\ell_p\oplus c_0)$ and $\mathcal{L}(\ell_1\oplus\ell_q)$, $1<p<2<q<\infty$, which are also based on arguments in \cite{SSTT07}.

\section{Closed subideals in complexified Banach spaces}

The main result in \cite{SZ14} was proved for real Banach spaces, and so our adaptation here will extend to complex Banach spaces via a complexification procedure.  For this reason, let us recall some facts about the {\bf complexification} $X_\mathbb{C}$ of a real Banach space $X$, defined as the Banach space
\[X_\mathbb{C}:=X\oplus iX,\]
and endowed with vector space operations
\[(x_1+iy_1)+(x_2+iy_2)=(x_1+x_2)+i(y_1+y_2)\]
and
\[(\alpha+i\beta)(x+iy)=(\alpha x-\beta y)+i(\beta x+\alpha y),\]
as well as the norm
\[\norm{x+iy}_{X_\mathbb{C}}:=\sup_{\phi\in[0,2\pi]}\norm{x\cos\phi+y\sin\phi}_X.\]
Notice that this means
\begin{equation}\label{complexification-l1}\frac{1}{2}\left(\norm{x}_X+\norm{y}_X\right)\leq\norm{x+iy}_{X_\mathbb{C}}\leq\norm{x}_X+\norm{y}_X\end{equation}
so that $x_j+iy_j\to x+iy$ in $X_\mathbb{C}$ if and only if $x_j\to x$ and $y_j\to y$ in $X$ (cf., e.g., \cite[pp5-6]{AA02}).  If $T\in\mathcal{L}(X,Y)$ is a continuous linear operator between Banach spaces $X$ and $Y$, we can consider its {\bf complexification} $T_\mathbb{C}\in\mathcal{L}(X_\mathbb{C},Y_\mathbb{C})$ defined by
\[T_\mathbb{C}(x_1+ix_2)=Tx_1+iTx_2.\]
In this case, $\norm{T_\mathbb{C}}=\norm{T}$ (cf., e.g., \cite[Lemma 1.7]{AA02}), and we can write
\[T_\mathbb{C}=\begin{bmatrix}T&0\\0&T\end{bmatrix},\]
where we view elements of $X_\mathbb{C}$ and $Y_\mathbb{C}$ as $1\times 2$ matrices, in the obvious way.  In fact, if $\mathcal{T}\in\mathcal{L}(X_\mathbb{C},Y_\mathbb{C})$ then there exist operators $R,S\in\mathcal{L}(X,Y)$ such that
\[\mathcal{T}=R_\mathbb{C}+iS_\mathbb{C}=\begin{bmatrix}R&-S\\S&R\end{bmatrix}\]
(cf., e.g., \cite[Theorem 1.8]{AA02}).

We now give some basic results about complexification of ideals.   If $\mathcal{A}\subseteq\mathcal{L}(X,Y)$ for real Banach spaces $X$ and $Y$ then we write
\[\mathcal{A}_\mathbb{C}:=\{R_\mathbb{C}:R\in\mathcal{A}\}.\]

\begin{proposition}\label{complexification-span}Let $X$ and $Y$ be real Banach spaces, and suppose $\mathcal{A}$ is a subset of $\mathcal{L}(X,Y)$.  Then
\[\text{span}(\mathcal{A}_\mathbb{C})=\left\{\begin{bmatrix}R&-S\\S&R\end{bmatrix}:R,S\in\text{span}(\mathcal{A})\right\}\]
and
\[[\mathcal{A}_\mathbb{C}]=\overline{\text{span}}(\mathcal{A}_\mathbb{C})=\left\{\begin{bmatrix}R&-S\\S&R\end{bmatrix}:R,S\in[\mathcal{A}]\right\}.\]\end{proposition}

\begin{proof}Notice that for all $\alpha,\beta\in\mathbb{R}$ and $R\in\text{span}(\mathcal{A})$ we have
\[(\alpha+i\beta)R_\mathbb{C}=\begin{bmatrix}\alpha R&-\beta R\\\beta R&\alpha R\end{bmatrix}.\]
Also, the set
\[\left\{\begin{bmatrix}R&-S\\S&R\end{bmatrix}:R,S\in\text{span}(\mathcal{A})\right\}\]
is clearly closed under addition.  This shows that
\[\text{span}(\mathcal{A}_\mathbb{C})\subseteq\left\{\begin{bmatrix}R&-S\\S&R\end{bmatrix}:R,S\in\text{span}(\mathcal{A})\right\}.\]
Notice that if $R=\sum_{j=1}^m\alpha_jR_j$ and $S=\sum_{k=1}^n\beta_kS_k$ for $\alpha_j,\beta_k\in\mathbb{R}$ and $R_j,S_k\in\mathcal{A}$ then
\[\begin{bmatrix}R&-S\\S&R\end{bmatrix}=\sum_{j=1}^m\alpha_j(R_j)_\mathbb{C}+i\sum_{k=1}^n\beta_k(S_k)_\mathbb{C},\]
which shows that the reverse inequality also holds.

Next, suppose $\mathcal{R}_j\to\mathcal{R}$ for $(\mathcal{R}_j)_{j=1}^\infty\subseteq\text{span}(\mathcal{A}_\mathbb{C})$.  Write
\[\mathcal{R}_j=\begin{bmatrix}R_j&-S_j\\S_j&R_j\end{bmatrix}\;\;\;\text{ and }\;\;\;\mathcal{R}=\begin{bmatrix}R&-S\\S&R\end{bmatrix}\]
for $R_j,S_j\in\text{span}(\mathcal{A})$ and $R,S\in\mathcal{L}(X,Y)$.  Suppose towards a contradiction that $R\notin[\mathcal{A}]$.  Then we can find $\epsilon>0$ such that $\norm{R_j-R}>\epsilon$ for all $j\in\mathbb{Z}^+$.  Hence, there exists $x_j\in X$ such that $\norm{(R_j-R)x_j}_Y\geq\epsilon\norm{x_j}_X$.  However, due to
\[\begin{bmatrix}(R_j-R)&-(S_j-S)\\(S_j-S)&(R_j-R)\end{bmatrix}\begin{bmatrix}x_j\\0\end{bmatrix}=\begin{bmatrix}(R_j-R)x_j\\(S_j-S)x_j\end{bmatrix},\]
we have
\[\norm{(\mathcal{R}_j-\mathcal{R})(x_j+i0)}_{Y_\mathbb{C}}\geq\frac{1}{2}\norm{(R_j-R)x_j}_Y\geq\frac{\epsilon}{2}\norm{x_j}_X=\frac{\epsilon}{2}\norm{x_j+i0}_{X_\mathbb{C}},\]
contradicting the fact that $\mathcal{R}_j\to\mathcal{R}$.  Thus, $R\in[\mathcal{A}]$, and an analogous argument shows that $S\in[\mathcal{A}]$ as well.  It follows that
\[[\mathcal{A}_\mathbb{C}]\subseteq\left\{\begin{bmatrix}R&-S\\S&R\end{bmatrix}:R,S\in[\mathcal{A}]\right\}.\]
The reverse inequality is even more obvious, and we are done.\end{proof}

\begin{proposition}\label{complexification-ideal}Let $W$, $X$, $Y$, and $Z$ be real Banach spaces, and let $T\in\mathcal{L}(W,Z)$.  Then
\[[\mathcal{G}_{T_\mathbb{C}}](X_\mathbb{C},Y_\mathbb{C})=\left\{\begin{bmatrix}R&-S\\S&R\end{bmatrix}:R,S\in[\mathcal{G}_T](X,Y)\right\}.\]\end{proposition}

\begin{proof}Let $\mathcal{T}\in\mathcal{G}_{T_\mathbb{C}}(X_\mathbb{C},Y_\mathbb{C})$, and write $\mathcal{T}=\mathcal{S}T_\mathbb{C}\mathcal{R}$ for
\[\mathcal{R}=\begin{bmatrix}R_1&-R_2\\R_2&R_1\end{bmatrix}\in\mathcal{L}(X_\mathbb{C},W_\mathbb{C})\;\;\;\text{ and }\;\;\;\mathcal{S}=\begin{bmatrix}S_1&-S_2\\S_2&S_1\end{bmatrix}\in\mathcal{L}(Z_\mathbb{C},Y_\mathbb{C}).\]
Then
\[\mathcal{T}=\begin{bmatrix}(S_1TR_1-S_2TR_2)&-(S_1TR_2+S_2TR_1)\\(S_1TR_2+S_2TR_1)&(S_1TR_1-S_2TR_2)\end{bmatrix}\]
so that
\[\mathcal{G}_{T_\mathbb{C}}(X_\mathbb{C},Y_\mathbb{C})\subseteq\left\{\begin{bmatrix}R&-S\\S&R\end{bmatrix}:R,S\in\text{span}\left(\mathcal{G}_T(X,Y)\right)\right\}.\]
By Proposition \ref{complexification-span} we now have
\[\mathcal{G}_{T_\mathbb{C}}(X_\mathbb{C},Y_\mathbb{C})\subseteq\text{span}\left(\left(\mathcal{G}_T(X,Y)\right)_\mathbb{C}\right)\]
and hence, applying Proposition \ref{complexification-span} once more,
\[[\mathcal{G}_{T_\mathbb{C}}](X_\mathbb{C},Y_\mathbb{C})\subseteq[\left(\mathcal{G}_T(X,Y)\right)_\mathbb{C}]=\left\{\begin{bmatrix}R&-S\\S&R\end{bmatrix}:R,S\in[\mathcal{G}_T](X,Y)\right\}.\]
To see the reverse inequality, suppose $R,S\in\text{span}(\mathcal{G}_T(X,Y))$, and write $R=\sum_{j=1}^m\alpha_jB_jTA_j$ and $S=\sum_{k=1}^n\beta_kD_kTC_k$ for $A_j,C_k\in\mathcal{L}(X,W)$ and $B_j,D_k\in\mathcal{L}(Z,Y)$.  Notice that
\[\begin{bmatrix}R&-S\\S&R\end{bmatrix}=\sum_{j=1}^m\alpha_j(B_j)_\mathbb{C}T_\mathbb{C}(A_j)_\mathbb{C}+i\sum_{k=1}^n\beta_k(D_k)_\mathbb{C}T_\mathbb{C}(C_k)_\mathbb{C}\in\text{span}\left(\mathcal{G}_{T_\mathbb{C}}(X_\mathbb{C},Y_\mathbb{C})\right).\]
If follows that
\[\left\{\begin{bmatrix}R&-S\\S&R\end{bmatrix}:R,S\in\text{span}\left(\mathcal{G}_T(X,Y)\right)\right\}\subseteq\text{span}\left(\mathcal{G}_{T_\mathbb{C}}(X_\mathbb{C},Y_\mathbb{C})\right).\]
Let us define $\mathcal{A}=\text{span}(\mathcal{G}_T(X,Y))$.  Thus, again applying Proposition \ref{complexification-span} successively we get
\begin{multline*}\left\{\begin{bmatrix}R&-S\\S&R\end{bmatrix}:R,S\in[\mathcal{G}_T](X,Y)\right\}=\left\{\begin{bmatrix}R&-S\\S&R\end{bmatrix}:R,S\in[\mathcal{A}]\right\}\\=[\mathcal{A}_\mathbb{C}]=[\text{span}(\mathcal{A}_\mathbb{C})]=\left[\left\{\begin{bmatrix}R&-S\\S&R\end{bmatrix}:R,S\in\text{span}(\mathcal{A})\right\}\right]\\=\left[\left\{\begin{bmatrix}R&-S\\S&R\end{bmatrix}:R,S\in\mathcal{A}\right\}\right]\\=\left[\left\{\begin{bmatrix}R&-S\\S&R\end{bmatrix}:R,S\in\text{span}\left(\mathcal{G}_T(X,Y)\right)\right\}\right]\subseteq[\text{span}\left(\mathcal{G}_{T_\mathbb{C}}(X_\mathbb{C},Y_\mathbb{C})\right)]\\=[\mathcal{G}_{T_\mathbb{C}}](X_\mathbb{C},Y_\mathbb{C}).\end{multline*}
\end{proof}

The following is almost certainly known to specialists, but we will provide a short proof for completeness.

\begin{proposition}\label{complexification-FSS}Let $X$ and $Y$ be real Banach spaces, and let $R\in\mathcal{L}(X,Y)$.  Then $R\in\mathcal{FSS}(X,Y)$ if and only if $R_\mathbb{C}\in\mathcal{FSS}(X_\mathbb{C},Y_\mathbb{C})$.\end{proposition}

\begin{proof}Suppose $R\in\mathcal{FSS}(X,Y)$.  Let us begin by showing that
\begin{equation}\label{R0-in-FSS}R\oplus 0\in\mathcal{FSS}(X\oplus_{\ell_1}X,Y\oplus_{\ell_1}Y).\end{equation}
Suppose $\delta>0$.  Then there is $n\in\mathbb{Z}^+$ such that for every $n$-dimensional subspace $E$ of $X$ there exists $e\in E$ such that $\norm{Re}_Y<\delta\norm{e}_Y$.  Let $\widetilde{E}$ be an $n$-dimensional subspace of $X\oplus_{\ell_1}X$, and let $(e_j\oplus e'_j)_{j=1}^n$ be a basis for $\widetilde{E}$.  If $(e_j)_{j=1}^n$ is a linearly dependent set then find $(\alpha_j)_{j=1}^n\in\mathbb{R}^n$, not all zero, such that $\sum_{j=1}^n\alpha_je_j=0$ and hence
\[\norm{(R\oplus 0)\sum_{j=1}^n\alpha_j(e_j\oplus e'_j)}_{X\oplus_{\ell_1}X}=0<\delta\norm{\sum_{j=1}^n\alpha_j(e_j\oplus e'_j)}_{Y\oplus_{\ell_1}Y}\neq 0.\]
Otherwise $(e_j)_{j=1}^n$ is linearly independent, in which case it spans an $n$-dimensional subspace $[e_j]_{j=1}^n$ so that there exist $(\alpha_j)_{j=1}^n\in\mathbb{R}^n$ satisfying
\begin{multline*}\norm{(R\oplus 0)\sum_{j=1}^n\alpha_j(e_j\oplus e'_j)}_{X\oplus_{\ell_1}X}=\norm{R\sum_{j=1}^n\alpha_je_j}_Y\\<\delta\norm{\sum_{j=1}^n\alpha_je_j}_X\leq\delta\norm{\sum_{j=1}^n\alpha_j(e_j\oplus e'_j)}_{Y\oplus_{\ell_1}Y}.\end{multline*}
This proves \eqref{R0-in-FSS}, and by a nearly identical argument it follows that also $0\oplus R$ is $\mathcal{FSS}$, and hence that
\begin{equation}\label{RR-in-FSS}R\oplus R=(R\oplus 0)+(0\oplus R)\in\mathcal{FSS}(X\oplus_{\ell_1}X,Y\oplus_{\ell_1}Y).\end{equation}

Now let's show that $R_\mathbb{C}\in\mathcal{FSS}(X_\mathbb{C},Y_\mathbb{C})$.  Select any $\epsilon>0$, and let $N\in\mathbb{Z}^+$ be such that for any $N$-dimensional subspace $\widetilde{E}$ of $X\oplus_{\ell_1}X$ there exists $\tilde{e}\in\widetilde{E}$ such that
\[\norm{(R\oplus R)\tilde{e}}_{Y\oplus_{\ell_1}Y}<\epsilon\norm{\tilde{e}}_{X\oplus_{\ell_1}X}.\]
Let $\widehat{E}$ be any $N$-dimensional subspace of $X_\mathbb{C}$, and let $(f_j+ig_j)_{j=1}^N$ be a basis for $\widehat{E}$.  By \eqref{complexification-l1}, $(f_j\oplus g_j)_{j=1}^N$ must be linearly independent in and hence $[f_j\oplus g_j]_{j=1}^N$ an $N$-dimensional subspace of $X\oplus_{\ell_1}X$.  So by \eqref{RR-in-FSS} we can find nonzero $(\beta_j)_{j=1}^N\in\mathbb{R}^N$ such that
\[\norm{(R\oplus R)\sum_{j=1}^N\beta_j(f_j\oplus g_j)}_{Y\oplus_{\ell_1}Y}<\frac{\epsilon}{2}\norm{\sum_{j=1}^N\beta_j(f_j\oplus g_j)}_{X\oplus_{\ell_1}X}.\]
and hence, together with \eqref{complexification-l1},
\begin{multline*}\norm{R_\mathbb{C}\sum_{j=1}^N\beta_j(f_j+ig_j)}_{Y\oplus_{\ell_1}Y}\leq\norm{(R\oplus R)\sum_{j=1}^N\beta_j(f_j\oplus g_j)}_{Y\oplus_{\ell_1}Y}\\<\frac{\epsilon}{2}\norm{\sum_{j=1}^N\beta_j(f_j\oplus g_j)}_{X\oplus_{\ell_1}X}\leq\epsilon\norm{\sum_{j=1}^N\beta_j(f_j+ig_j)}_{X_\mathbb{C}}.\end{multline*}

For the converse, let us suppose instead that $R_\mathbb{C}\in\mathcal{FSS}(X_\mathbb{C},Y_\mathbb{C})$.  Let $\epsilon>0$, and select $n\in\mathbb{Z}^+$ such that for any $n$-dimensional subspace $\widehat{E}$ of $X_\mathbb{C}$ there exists $e\in\widehat{E}$ such that $\norm{R_\mathbb{C}e}_{Y_\mathbb{C}}<\frac{\epsilon}{2}\norm{e}_{X_\mathbb{C}}$.  Let $E$ be any $n$-dimensional subspace of $X$, and find a basis $(e_j)_{j=1}^n$ for $E$.  Then $(e_j+i0)_{j=1}^n$ spans an $n$-dimensional subspace of $X_\mathbb{C}$, which means we can find $(\alpha_j)_{j=1}^n\in\mathbb{R}^n$ and $(\beta_j)_{j=1}^n\in\mathbb{R}^n$ such that
\[\norm{\sum_{j=1}^n(\alpha_j+i\beta_j)R_\mathbb{C}(e_j+i0)}_{Y_\mathbb{C}}<\frac{\epsilon}{2}\norm{\sum_{j=1}^n(\alpha_j+i\beta_j)(e_j+i0)}_{X_\mathbb{C}}\]
and hence, together with \eqref{complexification-l1},
\begin{multline*}\norm{R\sum_{j=1}^n\alpha_je_j}_Y+\norm{R\sum_{j=1}^n\beta_je_j}_Y\leq 2\norm{\sum_{j=1}^n(R\alpha_je_j+iR\beta_je_j)}_{Y_\mathbb{C}}\\=2\norm{\sum_{j=1}^n(\alpha_j+i\beta_j)R_\mathbb{C}(e_j+i0)}_{Y_\mathbb{C}}<\epsilon\norm{\sum_{j=1}^n(\alpha_j+i\beta_j)(e_j+i0)}_{X_\mathbb{C}}\\=\epsilon\norm{\sum_{j=1}^n(\alpha_je_j+i\beta_je_j)}_{X_\mathbb{C}}\leq\epsilon\norm{\sum_{j=1}^n\alpha_je_j}_X+\epsilon\norm{\sum_{j=1}^n\beta_je_j}_X.\end{multline*}
It follows that either
\[\norm{R\sum_{j=1}^n\alpha_je_j}_Y<\epsilon\norm{\sum_{j=1}^\infty\alpha_je_j}_X\;\;\;\text{ or }\;\;\;\norm{R\sum_{j=1}^n\beta_je_j}_Y<\epsilon\norm{\sum_{j=1}^n\beta_je_j}_X.\]
This means $R\in\mathcal{FSS}(X,Y)$.\end{proof}

\section{Continuum many closed ideals in $\mathcal{L}(X)$ and $\mathcal{L}(X^*)$}

In this section we will adapt the proof of \cite[Theorem 6]{SZ14} to a more general case.  This will require us to summarize and restate many of the preliminaries in that paper.

Fix $p\in(1,2)$, and let $p'\in(2,\infty)$ denote its conjugate, i.e. $\frac{1}{p}+\frac{1}{p'}=1$.  Also, let $\bv=(v_n)_{n=1}^\infty$ denote a sequence of values in $(0,1]$.  For each $n\in\mathbb{Z}^+$, in \cite[Section 2.4]{SZ14} was defined a real finite-dimensional Banach space $E_{p',v_n}^{(n)}=(\mathbb{R}^n,\norm{\cdot}_{p',v_n})$, according to the rule
\[\norm{(a_j)_{j=1}^n}_{p',v_n}=\norm{(a_j)_{j=1}^n}_{\ell_{p'}^n}\vee v_n\norm{(a_j)_{j=1}^n}_{\ell_2^n}.\]
Let $(e_j^{(p',\bv,n)})_{j=1}^n$ denote the canonical basis of $E_{p',v_n}^{(n)}$, and denote by $(e_j^{(p',\bv,n)*})_{j=1}^n$ the biorthogonal basis for its dual $E_{p',v_n}^{(n)*}$.  We can also fix, once and for all, a sequence $(f_j^{(p,\bv,n)})_{j=1}^n$ of independent, symmetric, 3-valued, random variables in $L_p$, satisfying $\norm{f_j^{(p,\bv,n)}}_{L_p}=1$ and $\norm{f_j^{(p,\bv,n)}}_{L_2}=\frac{1}{v_n}$, and then define the space $F_{p,v_n}^{(n)}=[f_j^{(p,\bv,n)}]_{j=1}^n$.  When $p$ and $\bv$ understood from context, we will simply write $f_j^{(n)}=f_j^{(p,\bv,n)}$, and refer to $(f_j^{(n)})_{j=1}^n$ as the {\bf canonical basis} for $F_{p,v_n}^{(n)}$.

In \cite[p5]{SZ14} it was observed that we can view spaces $F_{p,v_n}^{(n)}$ as subspaces of $\ell_p^{3^n}$.  This is because, since the vectors $(f_j^{(n)})_{j=1}^n$ are 3-valued, their span is a subspace of the span of characteristic functions on $3^n$ pairwise disjoint sets in $L_p$, whose span is in turn isometrically isomorphic to $\ell_p^{3^n}$.  In particular, we can view $(f_j^{(n)})_{j=1}^n$ as vectors in $\ell_p^{3^n}$.  It was also observed in \cite[p5, eq. (4)]{SZ14} that
\begin{multline}\label{4}\frac{1}{K_p}\norm{\sum_{j=1}^na_je_j^{(p',\bv,n)*}}_{E_{p',v_n}^{(n)*}}\leq\norm{\sum_{j=1}^na_jf_j^{(n)}}_{\ell_p^{3^n}}\leq\norm{\sum_{j=1}^na_je_j^{(p',\bv,n)*}}_{E_{p',v_n}^{(n)*}}\\\forall(a_j)_{j=1}^n\in\mathbb{R}^n,\end{multline}
where $K_p\in[1,\infty)$ is a constant depending only on $p$.  (Note that we will refer to these same constants $K_p$, $1<p<\infty$, throughout this section.)  Then was given, in \cite[Proposition 1]{SZ14}, that $(f_j^{(n)})_{j=1}^n$ is a normalized, 1-unconditional basis for $F_{p,v_n}^{(n)}$, that there exist projections $P_{p,v_n}^{(n)}\in\mathcal{L}(\ell_p^{3^n})$ onto $F_{p,v_n}^{(n)}$, $n\in\mathbb{Z}^+$, and that these projections are uniformly bounded by $K_p$.

We shall need another important fact about these spaces, which has already been proved in \cite{SZ14}.

\begin{lemma}[{\cite[Proposition 1(iii)]{SZ14}}]\label{1-iii}Let $p\in(1,2)$, and let $\bv=(v_n)$ be a sequence in $(0,1]$.  Then for each $n\in\mathbb{Z}^+$, each $1\leq k\leq n$, and each $A\subseteq\{1,\cdots,n\}$ satisfying $|A|=k$, we have
\[\frac{1}{K_p}\left(k^{1/p}\wedge\frac{k^{1/2}}{v_n}\right)\leq\norm{\sum_{j\in A}f_j^{(n)}}_{\ell_p^{3^n}}\leq k^{1/p}\wedge\frac{k^{1/2}}{v_n},\]
where $(f_j^{(n)})_{j=1}^n$ denotes the canonical basis for $F_{p,v_n}^{(n)}$.\end{lemma}

Let us now define the space
\[Y_{p,\bv}=\left(\bigoplus_{n=1}^\infty F_{p,v_n}^{(n)}\right)_{\ell_p},\]
which can be viewed as a $K_p$-complemented subspace of $\ell_p=(\oplus_{n=1}^\infty\ell_p^{3^n})_{\ell_p}$.  Indeed, there exists a $K_p$-projection $P_{p,\bv}\in\mathcal{L}(\ell_p)$ onto $Y_{p,\bv}$ defined by
\[P_{p,\bv}=\bigoplus_{n=1}^\infty P_{p,v_n}^{(n)}:\ell_p=\left(\bigoplus_{n=1}^\infty\ell_p^{3^n}\right)_{\ell_p}\to\left(\bigoplus_{n=1}^\infty\ell_p^{3^n}\right)_{\ell_p}=\ell_p.\]
When $p$ is understood from context, we will write $Y_{\bv}=Y_{p,\bv}$ and $P_{\bv}=P_{p,\bv}$.

If $X$ is a Banach space with dimension $d\in\mathbb{Z}^+\cup\{\infty\}$, with a fixed basis $(x_i)_{i=1}^d$, we set $N_d=\{1,\cdots,d\}$ if $d\in\mathbb{Z}^+$ and $N_d=\mathbb{Z}^+$ otherwise.  We then define the {\bf fundamental function} $\varphi_X:N_d\cup[1,d)\to\mathbb{R}$ and the {\bf lower fundamental function} $\lambda_X:N_d\cup[1,d)\to\mathbb{R}$ by the rules
\[\varphi_X(k)=\sup\left\{\norm{\sum_{i\in A}x_i}_X:A\subseteq N_d,|A|\leq k\right\}\;\;\;\text{ and}\]
\[\lambda_X(k)=\inf\left\{\norm{\sum_{i\in A}x_i}:A\subseteq N_d,|A|\geq k\right\}\]
for all $k\in N_d$, extending them each to $N_d\cup[1,d)$ via linear interpolation.  Notice that we obtain, immediately from Lemma \ref{1-iii}, the following.

%+ + +
%
%Let $A,B\subseteq N_d$ and $k\in N_{2d}$.  If $|A|+|B|\leq k$ then
%\[\norm{\sum_{j\in A}(x_j+i0)+\sum_{j\in B}(0+ix_j)}_{X_\mathbb{C}}\leq\norm{\sum_{j\in A}x_j}_X+\norm{\sum_{j\in B}x_j}_X\leq 2\varphi_X(k\wedge d),\]
%and if $|A|+|B|\geq k$ then either $|A|\geq k/2$ or else $|B|\geq k/2$, and hence
%\[\norm{\sum_{j\in A}(x_j+i0)+\sum_{j\in B}(0+ix_j)}_{X_\mathbb{C}}\geq\frac{1}{2}\left(\norm{\sum_{j\in A}x_j}_X+\norm{\sum_{j\in B}x_j}_X\right)\geq\frac{1}{2}\lambda_X(k/2\vee 1).\]
%This means
%\[\varphi_{X_\mathbb{C}}(k)\leq 2\varphi_X(k\wedge d)\;\;\;\forall\;k\in N_{2d}\cup[1,2d),\]
%and
%\[\lambda_{X_\mathbb{C}}(k)\geq\frac{1}{2}\lambda_X(k/2\vee 1)\;\;\;\forall\;k\in N_{2d}\cup[1,2d).\]
%
%
%+ + +

\begin{corollary}\label{1-iii-corollary}Let $p\in(1,2)$, and let $\bv=(v_n)$ be a sequence in $(0,1]$.  Then for any $n\in\mathbb{Z}^+$ and any $1\leq k\leq n$ we have
\[\varphi_{F_n}(k)\leq k^{1/p}\wedge\frac{k^{1/2}}{v_n},\]
%\;\;\;\text{ and }\;\;\;\lambda_{F_n}(k)=\frac{1}{K_p}\left(k^{1/p}\wedge\frac{k^{1/2}}{v_n}\right),\]
where we have used the notation $F_n=F_{p,v_n}^{(n)}$.\end{corollary}

As a matter of convention, if $\bv=(v_n)_{n=1}^\infty$ is a nonincreasing sequence in $(0,1]$, then we extend $\bv$ to all of $[0,\infty)$ by setting $v_0=1$ and $v_x=v_{\lfloor x\rfloor}$.  This allows us to restate a Lemma proved in \cite{SZ14}.

\begin{lemma}[{\cite[Lemma 3]{SZ14}}]\label{Lemma-3}Let $p\in(1,2)$, and let $\bv=(v_n)_{n=1}^\infty$ be a nonincreasing sequence in $(0,1]$ satisfying $v_n\geq n^{-\eta}$ for all $n\in\mathbb{Z}^+$, where $\eta=\frac{1}{p}-\frac{1}{2}$.  Then for each $k\in\mathbb{Z}^+$ we have
\[\lambda_{Y_{p,\bv}}(k)\geq\frac{\lfloor\sqrt{k/2}\rfloor}{v_{\lfloor\sqrt{k/2}\rfloor}K_p\sqrt{2}}.\]\end{lemma}

In order to prove our main result, we shall need two more Lemmas from \cite{SZ14}.

\begin{lemma}[{\cite[Lemma 4]{SZ14}}]\label{Lemma-4}Let $Y$ be an infinite-dimensional Banach space with a normalized, 1-unconditional basis $(f_j)_{j=1}^\infty$, and for each $m\in\mathbb{Z}^+$, let $G_m$ be an $m$-dimensional Banach space with a normalized, 1-unconditional basis $(g_i^{(m)})_{i=1}^m$, and let $B_m:G_m\to Y$ be a linear operator with $\norm{B_m}\leq 1$.  Assume that the conditions
\begin{equation}\label{8}\lim_{k\to\infty}\sup_{m\geq k}\frac{\varphi_{G_m}(k)}{k}=0,\;\;\;\text{ and}\end{equation}
\begin{equation}\label{9}\lim_{m\to\infty}\frac{\varphi_{G_m}(m)}{\lambda_Y(cm)}=0\;\;\;\text{ for all }c>0\end{equation}
are both satisfied.  Then
\[\lim_{m\to\infty}\frac{1}{m}\sum_{i=1}^m\norm{B_mg_i^{(m)}}_\infty=0,\]
where we define $\norm{y}_\infty=\sup_{j\in\mathbb{Z}^+}|y_j|$ for any $y=\sum_{j=1}^\infty y_jf_j\in Y$.\end{lemma}

Recall that when $1\leq p\leq 2$, the space $\ell_p$ has cotype 2 (cf., e.g., \cite[Theorem 6.2.14]{AK06}).  Throughout this section, we shall let $C_p$ denote the cotype-2 constant for $\ell_p$.

\begin{lemma}\label{Lemma-5}Suppose $p\in(1,2)$ and $q\in(p,\infty)$, and let $\bv=(v_n)_{n=1}^\infty$ be a sequence in $(0,1]$.  Fix any $n\in\mathbb{Z}^+$ and any $\sigma\in(0,1]$ satisfying $v_n\leq\sigma^{\frac{1}{2}-\frac{1}{p'}}$, and let $(f_j^{(n)})_{j=1}^n$ denote the canonical basis for $F_{p,v_n}^{(n)}$.  If $y=\sum_{j=1}^ny_jf_j^{(n)}$ satisfies
\[\norm{y}_{F_{p,v_n}^{(n)}}\leq 1,\;\;\;\text{ and }\;\;\;\sup_{1\leq j\leq n}|y_j|\leq\sigma,\]
then
\[\norm{(y_j)_{j=1}^n}_{\ell_2^n}^q\leq\max\{C_p^p,K_p^q\}\cdot\sigma^{\frac{q}{2}-\frac{p}{2}\wedge\frac{q}{2}-\frac{q}{p'}}\cdot\norm{y}_{F_{p,v_n}^{(n)}}^p.\]
Furthermore, we have
\[\norm{(y_j)_{j=1}^n}_{\ell_2^n}\leq K_p\sigma^{\frac{1}{2}-\frac{1}{p'}}.\]\end{lemma}

\begin{proof}The first part of this Lemma is just a restatement of \cite[Lemma 5]{SZ14}.  To prove the ``furthermore'' part, let us raise each side of the previous inequality to the $1/q$ power, to obtain
\[\norm{(y_j)_{j=1}^n}_{\ell_2^n}\leq\max\{C_p^{p/q},K_p\}\cdot\sigma^{\frac{1}{2}-\frac{p}{2q}\wedge\frac{1}{2}-\frac{1}{p'}}\cdot\norm{y}_{F_{p,v_n}^{(n)}}^{p/q}.\]
Taking the limit as $q\to\infty$, we now have
\[\norm{(y_j)_{j=1}^n}_{\ell_2^n}\leq\max\{1,K_p\}\cdot\sigma^{\frac{1}{2}\wedge\frac{1}{2}-\frac{1}{p'}}=K_p\sigma^{\frac{1}{2}-\frac{1}{p'}}.\]\end{proof}

Fix any nonincreasing sequence $\bv=(v_n)$ in $(0,1]$ and any $p\in(1,2)$.  In the proof to \cite[Lemma 5]{SZ14} it was observed that, in $\ell_p$, any normalized and 1-unconditional basis $C_p$-dominates the canonical basis of $\ell_2$.  (This is a straightforward consequence of $\ell_p$, $1\leq p\leq 2$, having cotype 2.)  In particular, the canonical basis of each $\ell_2^n$, $n\in\mathbb{Z}^+$, is $C_p$-dominated by canonical basis of $F_{p,v_n}^{(n)}$.  Thus, for any $q\in[p,\infty]$ we may define the formal identity operator
\[I_{p,\bv,q}:Y_{p,\bv}\to Z_q\]
such that $\norm{I_{p,\bv,q}}\leq C_p$.  When $p$ is understood from context, we will simply write $I_{\bv,q}=I_{p,\bv,q}$.  In case we need to consider Banach spaces over $\mathbb{C}$, as an abuse of notation we will write $I_{\bv,q}$ and $P_{\bv}$ in place of the complexification maps
\[(I_{\bv,q})_{\mathbb{C}}:(Y_{p,\bv})_{\mathbb{C}}\to(Z_q)_{\mathbb{C}}\;\;\;\text{ and }\;\;\;(P_{\bv})_\mathbb{C}:(\ell_p)_\mathbb{C}\to(\ell_p)_\mathbb{C}.\]
%Recall that $(\ell_p)_\mathbb{C}$ is isometrically isomorphic to the complex version of $\ell_p$ (cf., e.g., \cite[Exercise 3.2.5]{AA02}), so that in the complex case we may understand $P_{\bv}$ to be a projection from $\ell_p$ onto $(Y_{\bv})_\mathbb{C}$.

This gives us enough machinery to prove the next Lemma.  It is analogous to \cite[Theorem 6]{SZ14}, and the proof here is essentially the same, except making slight modifications where necessary.  Note that we include some gritty details which were omitted from the original proof in \cite{SZ14}.  (In fact, we shall continue to follow this policy of giving greater detail when later adapting results from \cite{SSTT07} in Sections 4 and 5, as we believe it enhances readability.)

\begin{lemma}\label{Theorem-6}Let $p\in(1,2)$ and $q\in(p,\infty]$, and let $\bv=(v_n)_{n=1}^\infty$ and $\bw=(w_n)_{n=1}^\infty$ be nonincreasing sequences in $(0,1]$.  Let $Y_{\bv}$, $Y_{\bw}$, $I_{\bv,q}$, $I_{\bw,q}$, $P_{\bv}$, and $P_{\bw}$ be as above.  Suppose $X$ is a real Banach space such that $\pi:X\to W$ is a bounded projection onto a subspace $W$ of $X$, such that there exists an isomorphism $U:W\to\ell_p$.  Suppose also that $Y$ is a real Banach space such that there exists a continuous linear embedding $J:Z_q\to Y$.  Assume $v_n\geq n^{-\eta}$ and $w_n\geq n^{-\eta}$ for all $n\in\mathbb{Z}^+$, where $\eta=\frac{1}{p}-\frac{1}{2}$.  Also assume that
\begin{equation}\label{12}\lim_{n\to\infty}\frac{v_{\sqrt{cn}}}{w_n}=0\;\;\;\text{ for all }c\in(0,1),\end{equation}
where we extend $(v_n)_{n=1}^\infty$ to $(v_x)_{x\in[0,\infty)}$ using the rule $v_0=1$ and $v_x=v_{\lfloor x\rfloor}$, as described above.  Then
\[JI_{\bw,q}P_{\bw}U\pi\notin[\mathcal{G}_{I_{\bv,q}}](X,Y)\;\;\;\text{ and }\;\;\;(JI_{\bw,q}P_{\bw}U\pi)_\mathbb{C}\notin[\mathcal{G}_{I_{\bv,q}}](X_\mathbb{C},Y_\mathbb{C}).\]\end{lemma}

\begin{proof}By Proposition \ref{complexification-ideal}, it is sufficient to consider the real case.  Let $(f_j^{(n)})_{j=1}^n$ denote the canonical basis for $F_n:=F_{p,v_n}^{(n)}$, and let $(g_j^{(n)})_{j=1}^n$ denote the canonical basis for $G_n:=F_{p,w_n}^{(n)}$, which we may view as vectors in $Y_{\bv}$ and $Y_{\bw}$, respectively.  Set $h_j^{(n)}:=JI_{\bw}P_{\bw}g_j^{(n)}=JI_{\bw,q}g_j^{(n)}\in Y$, so that
\[\{h_j^{(n)}:n\in\mathbb{Z}^+,j=1,\cdots,n\}\]
is a copy of the standard basis for $Z_q$ as embedded into $Y$.  Then let
\[\{h_j^{(n)*}:n\in\mathbb{Z}^+,j=1,\cdots,n\}\subseteq Y^*\]
denote their biorthogonal functionals, which are bounded by some constant $K\in[1,\infty)$ since the $h_j^{(n)}$'s are seminormalized.  For each $m\in\mathbb{Z}^+$, define a continuous linear functional $\Phi_m\in\mathcal{L}(X,Y)^*$ by the rule
\[\Phi_m(V)=\frac{1}{m}\sum_{i=1}^mh_i^{(m)*}(VU^{-1}g_i^{(m)}),\;\;\;V\in\mathcal{L}(X,Y),\]
where here we are viewing $Y_{\bw}$ as a subspace of $\ell_p=(\oplus\ell_p^{3^n})_{\ell_p}$ and hence the $g_i^{(m)}$'s as vectors in $\ell_p$.  Notice that these functionals are uniformly bounded by $K\norm{U^{-1}}$ so that they have a weak*-accumulation point $\Phi\in\mathcal{L}(X,Y)^*$.  Since $\Phi_m(JI_{\bw,q}P_{\bw}U\pi)=1$ for all $m\in\mathbb{Z}^+$ we have $\Phi(JI_{\bw,q}P_{\bw}U\pi)=1$ as well.

Now let $A\in\mathcal{L}(Z_q,Y)$ and $B\in\mathcal{L}(X,Y_{\bv})$ with $\norm{A}\leq 1$ and $\norm{B}\leq\frac{1}{\norm{U^{-1}}}$.  Eventually, we will show that
\begin{equation}\label{13}\lim_{m\to\infty}\Phi_m(AI_{\bv,q}B)=0.\end{equation}
From this it will follow that $\Phi(AI_{\bv,q}B)=0$, and hence $JI_{\bw,q}P_{\bw}U\pi\notin[\mathcal{G}_{I_{\bv,q}}](X,Y)$ as desired.

Let $B_m:U^{-1}G_m\to Y_{\bv}$ denote the restriction of $B$ to $U^{-1}G_m$.  We claim that
\begin{equation}\label{14}\lim_{m\to\infty}\frac{1}{m}\sum_{i=1}^m\norm{B_mU^{-1}g_i^{(m)}}_\infty=0,\end{equation}
where, as in Lemma \ref{Lemma-4}, we define
\[\norm{y}_\infty=\sup\left\{|y_{n,j}|:n\in\mathbb{Z}^+,j=1,\cdots,n\right\}\]
for any $y=\sum_{n=1}^\infty\sum_{j=1}^ny_{n,j}f_j^{(n)}\in Y_{\bv}$.  By Corollary \ref{1-iii-corollary} we have $\varphi_{G_m}(k)\leq k^{1/p}$ for all $1\leq k\leq m$, which means condition \eqref{8} in Lemma \ref{Lemma-4} holds.  Notice that $Y_{\bv}$ has a normalized 1-unconditional basis, since it is formed from the $\ell_p$-sum of spaces with normalized 1-unconditional bases.  Furthermore, $\norm{B_mU^{-1}}\leq 1$ for all $m\in\mathbb{Z}^+$.  Thus, if we can verify condition \eqref{9}, then we will be able to apply Lemma \ref{Lemma-4} to get \eqref{14}.  Indeed, by Lemma \ref{Lemma-3} we have
\[\lambda_{Y_{\bv}}(k)\geq\frac{\lfloor\sqrt{k/2}\rfloor}{K_p\sqrt{2}\cdot v_{\lfloor\sqrt{k/2}\rfloor}}\geq\frac{\sqrt{k/2}-1}{v_{\sqrt{k/2}}K_p\sqrt{2}}\]
for all $k\in\mathbb{Z}^+$.  Recall that $\lambda_{Y_{\bv}}$ is a nondecreasing function extended to $[1,\infty)$ via linear interpolation.  Since also for each $x\in[1,\infty)$ we have $\sqrt{\lfloor x\rfloor/2}\geq\lfloor\sqrt{x/2}\rfloor$ and hence
\[v_{\sqrt{\lfloor x\rfloor/2}}\leq v_{\lfloor\sqrt{x/2}\rfloor}=v_{\sqrt{x/2}},\]
this means
\[\lambda_{Y_{\bv}}(x)\geq\lambda_{Y_{\bv}}(\lfloor x\rfloor)\geq\frac{\sqrt{\lfloor x\rfloor/2}-1}{v_{\sqrt{\lfloor x\rfloor/2}}K_p\sqrt{2}}\geq\frac{\sqrt{(x-1)/2}-1}{v_{\sqrt{x/2}}K_p\sqrt{2}}.\]
Notice that
\[\frac{\sqrt{(x-1)/2}-1}{v_{\sqrt{x/2}}K_p\sqrt{2}}\cdot\frac{3K_pv_{\sqrt{x/2}}}{x^{1/2}}=\frac{3\sqrt{(x-1)/2}-3}{\sqrt{2x}}=\frac{3}{2}\sqrt{1-\frac{1}{x}}-\frac{3}{\sqrt{2x}}\to\frac{3}{2}\]
as $x\to\infty$.  Thus, there is $\gamma\in[1,\infty)$ such that
\[\lambda_{Y_{\bv}}(x)\geq\frac{x^{1/2}}{3K_pv_{\sqrt{x/2}}}\]
for all $x\in(\gamma,\infty)$.  We also have, again by Corollary \ref{1-iii-corollary}, that $\varphi_{G_m}(m)\leq w_m^{-1}m^{1/2}$ for all $m\in\mathbb{Z}^+$.  Thus, for any $c>0$ and sufficiently large $m$, we have
\[\frac{\varphi_{G_m}(m)}{\lambda_{Y_{\bv}}(cm)}\leq\frac{3K_pv_{\sqrt{cm/2}}}{(cm)^{1/2}}\cdot w_m^{-1}m^{1/2}=\frac{3K_p}{c^{1/2}}\cdot\frac{v_{\sqrt{cm/2}}}{w_m}.\]
Notice that since $(v_n)$ is nonincreasing, if $c\geq 1$ then $v_{\sqrt{cm/2}}\leq v_{\sqrt{(1/2)m}}$ and hence
\[\frac{\varphi_{G_m}(m)}{\lambda_{Y_{\bv}}(cm)}\leq\frac{3K_p}{c^{1/2}}\cdot\frac{v_{\sqrt{cm/2}}}{w_m}\leq\frac{3K_p}{c^{1/2}}\cdot\frac{v_{\sqrt{(1/2)m}}}{w_m}\to 0\]
as $m\to\infty$, by assumption \eqref{12}.  Otherwise, $c/2\in(0,1)$ and so again by \eqref{12} we have
\[\frac{\varphi_{G_m}(m)}{\lambda_{Y_{\bv}}(cm)}\leq\frac{3K_p}{c^{1/2}}\cdot\frac{v_{\sqrt{cm/2}}}{w_m}=\frac{3K_p}{c^{1/2}}\cdot\frac{v_{\sqrt{(c/2)m}}}{w_m}\to 0.\]
Thus, condition \eqref{9} of Lemma \ref{Lemma-4} is satisfied, and \eqref{14} follows.

Now let us now prove \eqref{13}.  Set $t=\frac{1}{2}-\frac{1}{p'}$, and fix an arbitrary $\varrho\in(0,1)$.  Notice that by \eqref{12} we have $v_n\to 0$, and so we can find $n_0\in\mathbb{N}$ such that $v_n\leq\varrho^t$ for all $n\geq n_0$.  Then
\begin{equation}\label{15}|\Phi_m(AI_{\bv,q}B)|=\frac{1}{m}\left|\sum_{i=1}^mh_i^{(m)*}(AI_{\bv,q}BU^{-1}g_i^{(m)})\right|\leq\frac{K}{m}\sum_{i=1}^m\norm{I_{\bv,q}B_mU^{-1}g_i^{(m)}}_{Z_q}.\end{equation}
Select $1\leq i\leq m$, and then write
\[B_mU^{-1}g_i^{(m)}=\sum_{n=1}^\infty\sum_{j=1}^nx_{n,j}f_j^{(n)}\]
for scalars $x_{n,j}\in\mathbb{R}$.  Due to the fact that the basis $\{f_j^{(n)}:n\in\mathbb{Z}^+,j=1,\cdots,n\}$ is normalized and 1-unconditional, we have
\[\norm{B_mU^{-1}g_i^{(m)}}_\infty\leq\norm{B_mU^{-1}g_i^{(m)}}_{Y_{\bv}}\leq\norm{g_i^{(m)}}_{Y_{\bw}}=1\]
(cf., e.g., \cite[Lemma 1.49]{AA02}).  Now set
\[\sigma_i^{(m)}:=\varrho\vee\norm{B_mU^{-1}g_i^{(m)}}_\infty\leq 1,\]
and observe that $v_n\leq\sigma_i^{(m)t}$ for $n\geq n_0$.  Furthermore, $\max_j|x_{n,j}|\leq\sigma_i^{(m)}$, so that the conditions of Lemma \ref{Lemma-5} are satisfied for $(x_{n,j})_{j=1}^n$ when $n\geq n_0$.

We complete the proof by separately considering two cases, where $q=\infty$ and then where $q\neq\infty$.

{\bf Case $\boldsymbol{q=\infty}$.}

By Lemma \ref{Lemma-5} we have
\[\left(\sum_{j=1}^n|x_{n,j}|^2\right)^{1/2}\leq K_p\sigma_i^{(m)t}\;\;\;\;\text{ for all }n\geq n_0.\]
Thus,
\begin{multline*}\norm{I_{\bv,\infty}B_mU^{-1}g_i^{(m)}}_{Z_\infty}=\sup_{n\in\mathbb{Z}^+}\left(\sum_{j=1}^n|x_{n,j}|^2\right)^{1/2}\\\leq K_p\sigma_i^{(m)t}\vee\sup_{n\leq n_0}\left(\sum_{j=1}^n|x_{n,j}|^2\right)^{1/2}\leq K_p\sigma_i^{(m)t}\vee n_0^{1/2}\norm{B_mU^{-1}g_i^{(m)}}_\infty.\end{multline*}
Combining the above with \eqref{15}, we now have
\begin{multline*}|\Phi_m(AI_{\bv,\infty}B)|\leq\frac{K}{m}\sum_{i=1}^m\norm{I_{\bv,\infty}B_mU^{-1}g_i^{(m)}}_{Z_\infty}\\\leq\frac{K}{m}\sum_{i=1}^m\left[K_p\sigma_i^{(m)t}\vee n_0^{1/2}\norm{B_mU^{-1}g_i^{(m)}}_\infty\right]\\\leq\frac{KK_p}{m}\sum_{i=1}^m\sigma_i^{(m)t}+\frac{Kn_0^{1/2}}{m}\sum_{i=1}^m\norm{B_mU^{-1}g_i^{(m)}}_\infty\end{multline*}
Notice that $t\in(0,1/2)$ and so $\xi\mapsto\xi^t$ is a concave function on $(0,\infty)$.  Thus we have
\begin{multline*}|\Phi_m(AI_{\bv,\infty}B)|\leq\frac{KK_p}{m}\sum_{i=1}^m\sigma_i^{(m)t}+\frac{Kn_0^{1/2}}{m}\sum_{i=1}^m\norm{B_mU^{-1}g_i^{(m)}}_\infty\\\leq KK_p\left(\frac{1}{m}\sum_{i=1}^m\sigma_i^{(m)}\right)^t+\frac{Kn_0^{1/2}}{m}\sum_{i=1}^m\norm{B_mU^{-1}g_i^{(m)}}_\infty.\end{multline*}
Letting $m\to\infty$, condition \eqref{14} now gives us
\begin{multline*}|\Phi_m(AI_{\bv,\infty}B)|\leq KK_p\left(\frac{1}{m}\sum_{i=1}^m\sigma_i^{(m)}\right)^t+\frac{Kn_0^{1/2}}{m}\sum_{i=1}^m\norm{B_mU^{-1}g_i^{(m)}}_\infty\\=KK_p\left(\frac{1}{m}\sum_{i=1}^m\varrho\vee\norm{B_mU^{-1}g_i^{(m)}}_\infty\right)^t+\frac{Kn_0^{1/2}}{m}\sum_{i=1}^m\norm{B_mU^{-1}g_i^{(m)}}_\infty\to KK_p\varrho^t.\end{multline*}
Since $\varrho\in(0,1)$ was arbitrary, we have \eqref{13}.

{\bf Case $\boldsymbol{q\neq\infty}$.}

By Lemma \ref{Lemma-5} we have
\[\left(\sum_{j=1}^n|x_{n,j}|^2\right)^{q/2}\leq N\sigma_i^{(m)r}\norm{\sum_{j=1}^nx_{n,j}f_j^{(n)}}_{F_n}^p\;\;\;\;\text{ for all }n\geq n_0,\]
where $N=\max\{C_p^p,K_p^q\}$ and $r=\min\{\frac{q}{2}-\frac{p}{2},\frac{q}{2}-\frac{q}{p'}\}$.  This gives us
\begin{multline*}\norm{I_{\bv,q}B_mU^{-1}g_i^{(m)}}_{Z_q}=\left(\sum_{n=1}^\infty\left(\sum_{j=1}^n|x_{n,j}|^2\right)^{q/2}\right)^{1/q}\\\leq\left(\sum_{n=1}^{n_0}\left(\sum_{j=1}^n|x_{n,j}|^2\right)^{q/2}\right)^{1/q}+\left(\sum_{n>n_0}\left(\sum_{j=1}^n|x_{n,j}|^2\right)^{q/2}\right)^{1/q}\\\leq\left(\sum_{n=1}^{n_0}\left(\norm{B_mU^{-1}g_i^{(m)}}_\infty n_0^{1/2}\right)^q\right)^{1/q}+\left(\sum_{n>n_0}N\sigma_i^{(m)r}\norm{\sum_{j=1}^nx_{n,j}f_j^{(n)}}_{F_n}^p\right)^{1/q}\\\leq n_0^{1/2+1/q}\norm{B_mU^{-1}g_i^{(m)}}_\infty+N^{1/q}\sigma_i^{(m)r/q}\left(\sum_{n>n_0}\norm{\sum_{j=1}^nx_{n,j}f_j^{(n)}}_{F_n}^p\right)^{1/q}\\\leq n_0^{1/2+1/q}\norm{B_mU^{-1}g_i^{(m)}}_\infty+N^{1/q}\sigma_i^{(m)r/q}\norm{B_mU^{-1}g_i^{(m)}}_{Y_{\bv}}^{p/q}\\\leq n_0^{1/2+1/q}\norm{B_mU^{-1}g_i^{(m)}}_\infty+N^{1/q}\sigma_i^{(m)r/q}.\end{multline*}
Combining this with \eqref{15} and the concavity of $\xi\mapsto\xi^{r/q}$ (which follows from the fact that $r/q<1$), we get
\begin{multline*}|\Phi_m(AI_{\bv,q}B)|\leq\frac{K}{m}\sum_{i=1}^m\norm{I_{\bv,q}B_mU^{-1}g_i^{(m)}}_{Z_q}\\\leq\frac{K}{m}\sum_{i=1}^m\left(n_0^{1/2+1/q}\norm{B_mU^{-1}g_i^{(m)}}_\infty+N^{1/q}\sigma_i^{(m)r/q}\right)\\=\frac{Kn_0^{1/2+1/q}}{m}\sum_{j=1}^m\norm{B_mU^{-1}g_j^{(m)}}_\infty+\frac{KN^{1/q}}{m}\sum_{i=1}^m\sigma_i^{(m)r/q}\\\leq\frac{Kn_0^{1/2+1/q}}{m}\sum_{j=1}^m\norm{B_mU^{-1}g_j^{(m)}}_\infty+KN^{1/q}\left(\frac{1}{m}\sum_{i=1}^m\sigma_i^{(m)}\right)^{r/q}.\end{multline*}
Letting $m\to\infty$, condition \eqref{14} now gives us
\begin{multline*}|\Phi_m(AI_{\bv,q}B)|\leq \frac{Kn_0^{1/2+1/q}}{m}\sum_{j=1}^m\norm{B_mU^{-1}g_j^{(m)}}_\infty+KN^{1/q}\left(\frac{1}{m}\sum_{i=1}^m\sigma_i^{(m)}\right)^{r/q}\\\to KN^{1/q}\varrho^{r/q}.\end{multline*}
Since $\varrho\in(0,1)$ was arbitrary, we have \eqref{13}.\end{proof}

We will also need a basic fact about the existence of preduals.  We provide a short proof in lieu of a direct reference.

\begin{proposition}\label{predual}Let $X$ and $Y$ be Banach spaces, and suppose $Y^*$ is reflexive.  Then every operator in $\mathcal{L}(Y^*,X^*)$ is the dual of an operator in $\mathcal{L}(X,Y)$.  More precisely, if $T\in\mathcal{L}(Y^*,X^*)$ and $Y^*$ is reflexive then $T=S^*$ for some $S\in\mathcal{L}(X,Y)$.\end{proposition}

\begin{proof}Recall that if an operator $T\in\mathcal{L}(Y^*,X^*)$ is weak*-to-weak* continuous, then it has a predual $S=T_*\in\mathcal{L}(X,Y)$ (cf., e.g., \cite[Theorem 3.1.11]{Me98}).  We will show that this condition is satisfied whenever $Y^*$ is reflexive.  Since $T$ is norm-to-norm continuous, it is therefore weak-to-weak continuous (cf., e.g., \cite[Theorem II.5]{Di84}).  However, every weak*-open set is also weak-open (cf., e.g., \cite[pp12-3]{Di84}), which means $T$ is weak-to-weak* continuous.  Since $Y^*$ is reflexive, the weak and weak* topologies on $Y^*$ coincide.  Thus, $T$ is weak*-to-weak* continuous as desired.\end{proof}

Let $M=\{m_1<m_2<\cdots\}$ be an infinite subset of $\mathbb{Z}^+$.  If $p\in(1,2)$ then we let $\eta=\frac{1}{p}-\frac{1}{2}$, and then define the sequence $\bw_M^p=(w_n)_{n=1}^\infty$ in $(0,1]$ as follows.  We set $w_1=1$ and $w_{2^{3^{m_k}}}=2^{-\eta k}$ for each $k\in\mathbb{Z}^+$, and extend to the rest of $\mathbb{Z}^+$ via linear interpolation.  Now we shall fix a chain $\mathcal{C}$, with cardinality of the continuum, of subsets of $\mathbb{Z}^+$ satisfying the property that if $N$ and $M$ lie in $\mathcal{C}$ then either $N\subseteq M$ and $|M\setminus N|=\infty$, or else $M\subseteq N$ and $|M\setminus N|=\infty$.  This is not hard to achieve.  For instance, given $r\in(0,1)$, let $(t_{r,n})_{n=1}^\infty$ be a strictly increasing sequence of rational numbers such that $t_{r,n}\to r$ as $n\to\infty$.  Let $f:\mathbb{Q}\to\mathbb{Z}^+$ be any injective map, and define $M_r=\{f(t_{s,n}):s\in(0,r),n\in\mathbb{Z}^+\}$ for each $r\in(0,1)$.  Then for any $0<r_1<r_2<1$ we have $M_{r_1}\subseteq M_{r_2}$ and $|M_{r_2}\setminus M_{r_1}|=\infty$.  If necessary, we will delete a maximal and minimal element from $\mathcal{C}$.

Thus, we can state and prove the following result.

\begin{theorem}\label{main1-utility}Let $p\in(1,2)$ and $q\in(p,\infty]$.  Suppose $X$ is a real Banach space containing a complemented copy of $\ell_p$, and that $Y$ is a real Banach space containing a copy of $\ell_q$ if $q\neq\infty$, or of $c_0$ if $q=\infty$.  Let $\mathcal{C}$ be a chain as described above.  Then
\[[\mathcal{G}_{I_{\bw_N^p,q}}](X,Y)\subsetneq[\mathcal{G}_{I_{\bw_M^p,q}}](X,Y)\subsetneq\mathcal{FSS}(X,Y)\]
and
\[[\mathcal{G}_{I_{\bw_N^p,q}}](X_\mathbb{C},Y_\mathbb{C})\subsetneq[\mathcal{G}_{I_{\bw_M^p,q}}](X_\mathbb{C},Y_\mathbb{C})\subsetneq\mathcal{FSS}(X_\mathbb{C},Y_\mathbb{C})\]
for all $M\subseteq N$ lying in $\mathcal{C}$.

If furthermore $X$ is reflexive, then
\[[\mathcal{G}^*_{I_{\bw_N^p,q}}](Y^*,X^*)\subsetneq[\mathcal{G}^*_{I_{\bw_M^p,q}}](Y^*,X^*)\subsetneq\mathcal{SSCS}(Y^*,X^*)\]
and
\[[\mathcal{G}_{I_{\bw_N^p,q}}^*](Y^*_\mathbb{C},X^*_\mathbb{C})\subsetneq[\mathcal{G}^*_{I_{\bw_M^p,q}}](Y^*_\mathbb{C},X^*_\mathbb{C})\subsetneq\mathcal{SSCS}(Y^*_\mathbb{C},X^*_\mathbb{C})\]
for all $M\subseteq N$ lying in $\mathcal{C}$.\end{theorem}

\begin{proof}It was shown in the proof of \cite[Theorem A]{SZ14} that if $M\subseteq N\subseteq\mathbb{Z}^+$ with $N\setminus M$ also infinite, then the sequences $\bv=\bw_N^p=(v_n)_{n=1}^\infty$ and $\bw=\bw_M^p=(w_n)_{n=1}^\infty$ satisfy condition \eqref{12} of Lemma \ref{Theorem-6}, as well as the condition that $v_n\geq n^{-\eta}$ and $w_n\geq n^{-\eta}$ for all $n\in\mathbb{Z}^+$, where $\eta=\frac{1}{p}-\frac{1}{2}$.  Recall from Pe\l czy\'{n}ski's Decomposition Theorem  (cf., e.g., \cite[p73]{LT77}) that $Z_q$ is isomorphic to $\ell_q$ when $q\neq\infty$, and if instead $q=\infty$ we have that $Z_\infty$ embeds into $c_0$ via the map $\theta$ defined in the introduction.  Thus, in either case, $Z_q$ embeds into $Y$, and the conditions of Lemma \ref{Theorem-6} are satisfied.

Let us now recall some observations from the proof of \cite[Corollary 7]{SZ14}.  Indeed, from condition \eqref{12} it follows that $v_n\leq w_n$ for sufficiently large $n$.  Together with condition \eqref{4}, this means the canonical basis of $F_{p,\bw}^{(n)}$ is $K_p$-dominated by that of $F_{p,\bv}^{(n)}$ for sufficiently large $n$.  It follows that the formal inclusion map
\[I_{Y_{\bv},Y_{\bw}}:Y_{\bv}\to Y_{\bw}\]
is bounded.  Together with $I_{\bv,q}=I_{\bw,q}I_{Y_{\bv},Y_{\bw}}$, we get
\[[\mathcal{G}_{I_{\bv,q}}](X,Y)\subseteq[\mathcal{G}_{I_{\bw,q}}](X,Y),\]
and, by Proposition \ref{complexification-ideal},
\[[\mathcal{G}_{I_{\bv,q}}](X_\mathbb{C},Y_\mathbb{C})\subseteq[\mathcal{G}_{I_{\bw,q}}](X_\mathbb{C},Y_\mathbb{C}).\]
These inclusions are seen to be strict by applying Lemma \ref{Theorem-6}.

We also need to observe that $I_{\bw,q}$ is class $\mathcal{FSS}$, from which will follow the inclusions
\[[\mathcal{G}_{I_{\bw,q}}](X,Y)\subsetneq\mathcal{FSS}(X,Y),\]
and
\[[\mathcal{G}_{I_{\bw,q}}](X_\mathbb{C},Y_\mathbb{C})\subsetneq\mathcal{FSS}(X_\mathbb{C},Y_\mathbb{C}).\]
(These inclusions will be strict since we have deleted a maximal element from $\mathcal{C}$.)  Indeed, it has already been shown in \cite[Proposition 8]{SZ14} that the real version of $I_{\bw,q}$ is $\mathcal{FSS}$ when $q\neq\infty$, and in case $q=\infty$ then we see that the real version of $I_{\bw,\infty}$ is still $\mathcal{FSS}$ since it factors through the $\mathcal{FSS}$ map $I_{\bw,2}$.  Applying Proposition \ref{complexification-FSS} covers the complexification case.

Finally, let us consider the case where $X$ is reflexive.  Notice that $Y^{**}$ contains a copy of $Y$ and hence of $Z_q$, satisfying the conditions of Lemma \ref{Theorem-6} for $Y^{**}$ in place of $Y$.  Hence, there exist operators
\[T\in\mathcal{G}_{I_{\bw,q}}(X,Y^{**})\setminus[\mathcal{G}_{I_{\bv,q}}](X,Y^{**})\]
and
\[\widehat{T}\in\mathcal{G}_{I_{\bw,q}}(X_\mathbb{C},Y^{**}_\mathbb{C})\setminus[\mathcal{G}_{I_{\bv,q}}](X_\mathbb{C},Y_\mathbb{C}^{**}).\]
Since $X$ and hence also $X_\mathbb{C}$ are reflexive, by Proposition \ref{predual} we have $T=S^*$ for some $S\in\mathcal{L}(Y^*,X^*)$ and $\widehat{T}=\widehat{S}^*$ for some $\widehat{S}\in\mathcal{L}(Y_\mathbb{C}^*,X_\mathbb{C}^*)$.  It follows that
\[S\in\mathcal{G}_{I_{\bw,q}}^*(Y^*,X^*)\setminus[\mathcal{G}_{I_{\bv,q}}]^*(Y^*,X^*)\]
and
\[\widehat{S}\in\mathcal{G}_{I_{\bw,q}}^*(Y^*_\mathbb{C},X^*_\mathbb{C})\setminus[\mathcal{G}_{I_{\bv,q}}]^*(Y_\mathbb{C}^*,X_\mathbb{C}^*).\]
We must also have $S\notin[\mathcal{G}_{I_{\bv,q}}^*](Y^*,X^*)$, otherwise there would be $S_n\in\mathcal{G}_{I_{\bv,q}}^*(Y^*,X^*)$ such that $S_n\to S$ and hence $S_n^*\to S^*$ in norm, contradicting the fact that $S^*=T\notin[\mathcal{G}_{I_{\bv,q}}](X,Y^{**})$.  For the same reason, $\widehat{S}\notin[\mathcal{G}_{I_{\bv,q}}^*](Y^*_\mathbb{C},X^*_\mathbb{C})$.  It follows that the inclusions
\[[\mathcal{G}_{I_{\bv,q}}^*](Y^*,X^*)\subseteq[\mathcal{G}_{I_{\bw,q}}^*](Y^*,X^*),\]
and
\[[\mathcal{G}_{I_{\bv,q}}^*](Y^*_\mathbb{C},X^*_\mathbb{C})\subseteq[\mathcal{G}_{I_{\bw,q}}^*](Y^*_\mathbb{C},X^*_\mathbb{C})\]
are both strict.  The proof is then complete as we consider the full duality between $\mathcal{FSS}$ and $\mathcal{SSCS}$ (cf., e.g., \cite[Theorem 4]{Pl04}), together with the fact that we deleted a maximal and minimal element from $\mathcal{C}$.\end{proof}

Before proving the main Theorem \ref{main1}, let us study the relationship between closed subideals of $\mathcal{L}(X,Y)$ and closed ideals in $\mathcal{L}(X\oplus Y)$.  Recall again that if $\mathcal{A}$ is a subset of $\mathcal{L}(W,Z)$ then we denote by $\mathcal{G}_\mathcal{A}(X,Y)$ the set of all operators in $\mathcal{L}(X,Y)$ factoring through some operator in $\mathcal{A}$.

\begin{proposition}\label{order-isomorphism}Let $X$ and $Y$ denote Banach spaces.  For each closed subideal $\mathcal{I}$ in $\mathcal{L}(X,Y)$, we define
\[\Psi(\mathcal{I}):=[\mathcal{G}_\mathcal{I}](X\oplus Y),\]
the closed linear span of operators acting on $X\oplus Y$ and factoring through elements of $\mathcal{I}$.  Then $\Psi$ is an order isomorphism between the closed subideals in $\mathcal{L}(X,Y)$ and the closed ideals in $\mathcal{L}(X\oplus Y)$ of the form $\Psi(\mathcal{I})$.\end{proposition}

\begin{proof}Let $\mathcal{I}$ and $\mathcal{J}$ be closed subideals in $\mathcal{L}(X,Y)$.  Clearly, if $\mathcal{I}\subseteq\mathcal{J}$, then $\Psi(\mathcal{I})\subseteq\Psi(\mathcal{J})$.  Now let us suppose instead that $\Psi(\mathcal{I})\subseteq\Psi(\mathcal{J})$.  Pick any $T\in\mathcal{I}$.  Let $P:X\oplus Y\to X$ and $R:X\oplus Y\to Y$ denote the canonical projections onto $X$ and $Y$, respectively, and let $J:X\to X\oplus Y$ and $Q:Y\to X\oplus Y$ denote the canonical embeddings.  Then $QTP\in\Psi(\mathcal{I})\subseteq\Psi(\mathcal{J})$, and so we can find a sequence of finite sums satisfying
\[\lim_{n\to\infty}\sum_{j=1}^{m_n}B_{n,j}T_{n,j}A_{n,j}=QTP,\]
where $A_{n,j}\in\mathcal{L}(X\oplus Y,X)$, $B_{n,j}\in\mathcal{L}(Y,X\oplus Y)$, and $T_{n,j}\in\mathcal{J}$ for all $n$ and $j$.  Let us set
\[S_n:=\sum_{j=1}^{m_n}RB_{n,j}T_{n,j}A_{n,j}J\in\mathcal{J}.\]
Then $S_n\to RQTPJ=T$, and since $\mathcal{J}$ is closed we get $T\in\mathcal{J}$, showing that $\mathcal{I}\subseteq\mathcal{J}$.\end{proof}

We also need the following result on finding copies of $\ell_q$ or $c_0$ in a decomposed space.

\begin{proposition}\label{copy-of-Z}Let $1\leq p<q<\infty$.  Suppose $Y$ is a (real or complex) Banach space such that $\ell_p\oplus Y$ contains a copy of $\ell_q$ (resp. $c_0$).  Then $Y$ contains a copy of $\ell_q$ (resp. $c_0$).\end{proposition}

\begin{proof}Let $(x_n\oplus y_n)_{n=1}^\infty$ be a seminormalized basic sequence which is $K$-equivalent, for some $1\leq K<\infty$, to the canonical basis of $\ell_q$ (resp. $c_0$), where $x_n\in\ell_p$ and $y_n\in Y$ for all $n\in\mathbb{Z}^+$.  We claim that $(x_n)_{n=1}^\infty$ has a convergent subsequence.  Otherwise, by Rosenthal's $\ell_1$ Theorem, we consider separately the case where $(x_n)_{n=1}^\infty$ contains a subsequence equivalent to the canonical basis of $\ell_1$, which can only be true if $p=1$ since $\ell_p$ contains no copy of $\ell_1$ for $1<p<\infty$.  Pass to this subsequence, and set $x'_n=x_{2n+1}-x_{2n}$ so that $(x'_n)_{n=1}^\infty$ is a seminormalized basic sequence and equivalent to the canonical basis of $\ell_p=\ell_1$ in this case.  Next, consider the case where $(x_n)_{n=1}^\infty$ fails to contain a subsequence equivalent to the canonical basis of $\ell_1$.  Then we can pass to a subsequence if necessary and again define $x'_n=x_{2n+1}-x_{2n}$ so that $(x'_n)_{n=1}^\infty$ is seminormalized and weakly null.  By the Bessaga-Pe\l czy\'{n}ski Selection Principle together with \cite[Lemma 2.1.1 and Remark 2.1.2]{AK06}, we can pass to a further subsequence if necessary so that $(x'_n)_{n=1}^\infty$ is again equivalent to the canonical basis of $\ell_p$.  Thus in either case, we have passed to a subsequence so that $(x'_n)_{n=1}^\infty$ is $C$-equivalent to the canonical $\ell_p$ basis for some $1\leq C<\infty$.  Now set $y'_n:=y_{2n+1}-y_{2n}$.  Then $(x'_n\oplus y'_n)_{n=1}^\infty$ is $C'$-equivalent to the canonical $\ell_p$ basis for some $1\leq C'<\infty$.  Without loss of generality we may assume that $\ell_p\oplus Y$ is endowed with the $\ell_1$ norm, i.e. $\norm{x\oplus y}=\norm{x}+\norm{y}$ whenever $x\in\ell_p$ and $y\in Y$.  In the $\ell_q$ case we now have
\[C'N^{1/q}\geq\norm{\sum_{n=1}^N x'_n\oplus y'_n}\geq\norm{\sum_{n=1}^Nx'_n}\geq C^{-1}N^{1/p}\]
for all $N\in\mathbb{Z}^+$, which is impossible.  Similarly, in the $c_0$ case we have
\[C'\geq\norm{\sum_{n=1}^N x'_n\oplus y'_n}\geq\norm{\sum_{n=1}^Nx'_n}\geq C^{-1}N^{1/p}\]
for all $N\in\mathbb{Z}^+$, which is again impossible.  Thus, $(x_n)_{n=1}^\infty$ contains a convergent subsequence as claimed.  Pass to it, and let $x\in\ell_p$ be such that $\norm{x_n-x}\leq 2^{-n}$ for all $n\in\mathbb{Z}^+$.  Then
\[\norm{\sum_{n=1}^Nx_n\oplus y_n}\geq\norm{\sum_{n=1}^Nx_n}\geq\norm{\sum_{n=1}^Nx}-\sum_{n=1}^N\norm{x_n-x}\geq\norm{\sum_{n=1}^Nx}-1=N\norm{x}-1.\]
In the $\ell_q$ case this means
\[KN^{1/q}\geq\norm{\sum_{n=1}^Nx_n\oplus y_n}\geq N\norm{x}-1\]
for all $N\in\mathbb{Z}^+$, and in the $c_0$ case we have
\[K\geq\norm{\sum_{n=1}^Nx_n\oplus y_n}\geq N\norm{x}-1\]
for all $N\in\mathbb{Z}^+$.  Either way, we must have $x=0$ or face a contradiction as $N\to\infty$.  Thus, by the Principle of Small Perturbations, we can pass to a subsequence if necessary so that $(x_n\oplus y_n)_{n=1}^\infty$ is equivalent to $(0\oplus y_n)_{n=1}^\infty$.  It follows that $(y_n)_{n=1}^\infty$ is equivalent to the canonical basis of $\ell_q$ (resp. $c_0$), and hence that $Y$ contains a copy of $\ell_q$ (resp. $c_0$).\end{proof}

Now we can proceed with the proof of the main Theorem.

\begin{proof}[Proof of Theorem \ref{main1}]Let $X$ be a real Banach space containing a complemented copy of $\ell_p$, and a copy of either $\ell_q$ or $c_0$.  Let us decompose $X=\ell_p\oplus Y$ for some subspace $Y$ of $X$.  Notice that by Proposition \ref{copy-of-Z}, $Y$ contains a copy of either $\ell_q$ or $c_0$.

Next, let $\Psi$ be as defined in Proposition \ref{order-isomorphism}.  Since $\mathcal{FSS}$ and $\mathcal{SSCS}$ are closed operator ideals, we have
\[\Psi(\mathcal{FSS}(\ell_p,Y))\subseteq\mathcal{FSS}(\ell_p\oplus Y)=\mathcal{FSS}(X)\]
and
\[\Psi(\mathcal{SSCS}(\ell_p,Y))\subseteq\mathcal{SSCS}(\ell_p\oplus Y)=\mathcal{SSCS}(X).\]
Similarly, since $(\ell_p)_\mathbb{C}\oplus Y_\mathbb{C}=(\ell_p\oplus Y)_\mathbb{C}$, we have
\[\Psi(\mathcal{FSS}((\ell_p)_\mathbb{C},Y_\mathbb{C}))\subseteq\mathcal{FSS}(X_\mathbb{C})\;\;\;\text{ and }\;\;\;\Psi(\mathcal{SSCS}((\ell_p)_\mathbb{C},Y_\mathbb{C}))\subseteq\mathcal{SSCS}(X_\mathbb{C}).\]
Due to $p\in(1,2)$, the space $\ell_p$ is reflexive.  Applying Theorem \ref{main1-utility} and Proposition \ref{order-isomorphism} therefore completes the proof.\end{proof}

\section{Incomparable ideals in $\mathcal{L}(X)$ and $\mathcal{L}(X^*)$}

In this section we prove Theorem \ref{main2}.  Note once more that we will very closely follow the proof of \cite[Theorem 5.4]{SSTT07}, except making certain modifications where necessary.  We will need the following preliminary, which was given in \cite{SSTT07} as a Corollary to \cite[Theorem 9.13]{DJT95}.  As a matter of notation, if $A$ is an $n\times n$ matrix, then we let $\norm{A}_{p,q}$, $1\leq p,q\leq\infty$, denote the operator norm when $A$ is viewed as an operator $A:\ell_p^n\to\ell_q^n$.

\begin{proposition}[{\cite[Corollary 5.2]{SSTT07}}]\label{5.2}Let $m\in\mathbb{Z}^+$, and let $1\leq p<r<q\leq\infty$.  Suppose $U$ is an invertible $m\times m$ matrix satisfying $\norm{U}_{p,q}\leq 1$ and $\norm{U}_{r',r'}\leq\delta$.  Then for any factorization $U=AB$ we must have $\norm{A}_{r,q}\norm{B}_{p,r}\geq\delta^{-1}$.  For $i=1,\cdots,m$, let $e_i\in\mathbb{K}^m$ denote the $i$th coordinate vector.  If $\widetilde{U}$ is another $m\times m$ matrix satisfying
\[\norm{U-\widetilde{U}}_{p,q}\leq\left(2\max_{1\leq i\leq m}\norm{U^{-1}e_i}_p\right)^{-1}\]
then for any factorization $\widetilde{U}=\widetilde{A}\widetilde{B}$ we must have
\[\norm{\widetilde{A}}_{r,q}\norm{\widetilde{B}}_{p,r}\geq(2\delta)^{-1}.\]\end{proposition}

\noindent This is enough to prove the next result.

\begin{theorem}\label{Hadamard}Let $1<p<2$, and let $p'$ be its conjugate, i.e. $\frac{1}{p}+\frac{1}{p'}=1$.  Suppose $X$ is a (real or complex) Banach space containing a complemented copy of $\ell_p$, and let $P:X\to\ell_p$ denote a projection onto $\ell_p$.
\begin{itemize}\item[(i)]  Suppose $Y$ is a (real or complex) Banach space containing a copy of $c_0$, and let $J:c_0\to Y$ be any bounded linear embedding.  Then there exists an operator $U\in(\mathcal{FSS}\cap\mathcal{SSCS})(\ell_p,c_0)$ such that
\[JUP\notin[\mathcal{G}_{\ell_2}](X,Y).\]
Furthermore, if $X$ is reflexive then there exists an operator
\[V\in(\mathcal{FSS}\cap\mathcal{SSCS})(Y^*,X^*)\setminus[\mathcal{G}_{\ell_2}](Y^*,X^*).\]
\item[(ii)]  Suppose $\widehat{Y}$ is a (real or complex) Banach space containing a copy of $\ell_q$, $q\in[p',\infty)$, and let $\widehat{J}:\ell_q\to\widehat{Y}$ be any bounded linear embedding.  Then there exists an operator $\widehat{U}\in(\mathcal{FSS}\cap\mathcal{SSCS})(\ell_p,\ell_q)$ such that
\[\widehat{J}\widehat{U}P\notin[\mathcal{G}_{\ell_2}](X,\widehat{Y}).\]
Furthermore, if $X$ is reflexive then there exists an operator
\[\widehat{V}\in(\mathcal{FSS}\cap\mathcal{SSCS})(\widehat{Y}^*,X^*)\setminus[\mathcal{G}_{\ell_2}](\widehat{Y}^*,X^*).\]\end{itemize}\end{theorem}

\begin{proof}Let us inductively define a sequence $(H_n)_{n=1}^\infty$ of $2^{n-1}\times 2^{n-1}$ matrices.  Set $H_1:=[1]$, and if $H_n$ has been defined for $n\in\mathbb{Z}^+$, set
\[H_{n+1}:=\begin{bmatrix}H_n&H_n\\H_n&-H_n\end{bmatrix}.\]
Note that, in the literature, each $H_n$, $n\in\mathbb{Z}^+$ is called the $n$th {\it Hadamard matrix}, and for any $1\leq r\leq s\leq\infty$ it can be viewed as an operator $H_n\in\mathcal{L}(\ell_r^{2^{n-1}},\ell_s^{2^{n-1}})$.

It was observed in \cite[Remark 5.3]{SSTT07} (and is routine to verify) that each
\[H_n^2=2^{n-1}I_{2^{n-1}},\]
where $I_{2^{n-1}}$ denotes the $2^{n-1}\times 2^{n-1}$ identity matrix, and that
\[\norm{H_n}_{2,2}=2^{(n-1)/2}\;\;\;\text{ and }\;\;\;\norm{H_n}_{1,\infty}=1.\]
Thus, if $r\in(1,2)$ then we can apply the Riesz-Thorin Interpolation Theorem with $\phi=2/r'\in[0,1]$, to obtain
\[\norm{H_n}_{r,r'}\leq\norm{H_n}_{2,2}^\phi\norm{H_n}_{1,\infty}^{1-\phi}=2^{(n-1)/r'}.\]
We can therefore define, for any $r\in[1,2]$,
\[U_n^{(r)}:=2^{-(n-1)/r'}H_n\]
so that $\norm{U_n^{(p)}}_{p,p'}\leq 1$, and hence also $\norm{U_n^{(p)}}_{p,\infty}\leq 1$.  Due to these facts, we can define the norm-1 linear operator
\[U^{(p)}=\bigoplus_{n=1}^\infty U_n^{(p)}:\left(\bigoplus_{n=1}^\infty\ell_p^{2^{n-1}}\right)_{\ell_p}\to\left(\bigoplus_{n=1}^\infty\ell_{p'}^{2^{n-1}}\right)_{\ell_{p'}}.\]
Next, define
\[\widehat{U}:=I_{p',q}U^{(p)}\;\;\;\text{ and }\;\;\;U:=I_{p',0}U^{(p)}.\]
Since $I_{p',0}$ is $\mathcal{FSS}$, it follows that $U$ is as well.  That $\widehat{U}$ is $\mathcal{FSS}$ has already been shown in \cite[Theorem 6.8]{SSTT07}.

Since $\ell_\infty$ is injective, it will help to first consider that space.  Suppose towards a contradiction that $I_{q,\infty}\widehat{U}\in[\mathcal{G}_{\ell_2}](\ell_p,\ell_\infty)$.  Recall that if $X$, $Y$, and $Z$ are Banach spaces with $Z\cong Z\oplus Z$, then $\mathcal{G}_Z(X,Y)$ is always a linear space (cf., e.g., \cite[eq.(2),p313]{Sc12}).  In particular, $\mathcal{G}_{\ell_2}(\ell_p,\ell_\infty)$ is a linear space, and so there is $\widetilde{U}\in\mathcal{G}_{\ell_2}(\ell_p,\ell_\infty)$ with $\norm{I_{q,\infty}\widehat{U}-\widetilde{U}}_{p,\infty}<\frac{1}{2}$.  Write $\widetilde{U}=AB$ for $A\in\mathcal{L}(\ell_2,\ell_\infty)$ and $B\in\mathcal{L}(\ell_p,\ell_2)$, and set $C:=\norm{A}\norm{B}$.  Due to $2<p'<\infty$, we can pick $n\in\mathbb{Z}^+$ such that
\[C<\left(2\cdot 2^{(n-1)(\frac{1}{p'}-\frac{1}{2})}\right)^{-1}\]
Denote by $J_n:\ell_p^{2^{n-1}}\to(\bigoplus_{n=1}^\infty\ell_p^{2^{n-1}})_{\ell_p}$ and $R_n:(\bigoplus_{n=1}^\infty\ell_\infty^{2^{n-1}})_{\ell_\infty}\to\ell_\infty^{2^{n-1}}$ the canonical embedding and projection, which are both norm-1.  Now let $E$ be any $2^{n-1}$-dimensional subspace of $\ell_2$ containing $BJ_n\ell_p^{2^{n-1}}=\text{Im}(BJ_n)$.  Recall that a closed subspace of a Hilbert space is again a Hilbert space, and so by Parseval's identity we can now see that $E$ is isometrically isomorphic to $\ell_2^{2^{n-1}}$.  Hence, $R_n\widetilde{U}J_n=(R_nA|_E)(BJ_n)$ factors through $\ell_2^{2^{n-1}}$, and
\begin{equation}\label{factor-estimate}\norm{R_nA|_E}\norm{BJ_n}\leq\norm{A}\norm{B}=C<\left(2\cdot 2^{(n-1)(\frac{1}{p'}-\frac{1}{2})}\right)^{-1}.\end{equation}
Due to $H_n^2=2^{n-1}I_{2^{n-1}}$, we have $(U_n^{(p)})^{-1}=2^{-(n-1)}2^{(n-1)/p'}H_n$.  Recall also that $\norm{H_n}_{1,\infty}=1$, so that for each $i$th coordinate vector $e_i\in\mathbb{K}^{2^{n-1}}$ we have
\begin{multline*}\norm{(U_n^{(p)})^{-1}e_i}_p=2^{-(n-1)}2^{(n-1)/p'}\norm{H_ne_i}_p\leq 2^{-(n-1)}2^{(n-1)/p'}2^{(n-1)/p}\norm{H_ne_i}_\infty\\=\norm{H_ne_i}_\infty\leq \norm{H_n}_{1,\infty}\norm{e_i}_1=1,\end{multline*}
and hence
\begin{multline*}\norm{U_n^{(p)}-R_n\widetilde{U}J_n}_{p,\infty}=\norm{R_n(I_{q,\infty}\widehat{U}-\widetilde{U})J_n}_{p,\infty}\\\leq\norm{I_{q,\infty}\widehat{U}-\widetilde{U}}_{p,\infty}<\frac{1}{2}\leq\left(2\max_{1\leq i\leq 2^{n-1}}\norm{(U_n^{(p)})^{-1}e_i}_p\right)^{-1}.\end{multline*}
This gives us
\begin{multline*}\delta:=\norm{(U_n^{(p)})^{-1}}_{2,2}=2^{-(n-1)}2^{(n-1)/p'}\norm{H_n}_{2,2}\\=2^{-(n-1)}2^{(n-1)/p'}2^{(n-1)/2}=2^{(n-1)(\frac{1}{p'}-\frac{1}{2})}.\end{multline*}
Thus we have $\norm{U_n^{(p)}}_{p,\infty}\leq 1$, $\norm{(U_n^{(p)})^{-1}}_{2,2}=\delta$, $R_n\widetilde{U}J_n=(R_nA|_E)(BJ_n)$, and
\[\norm{U_n^{(p)}-R_n\widetilde{U}J_n}_{p,\infty}\leq\left(2\max_{1\leq i\leq 2^{n-1}}\norm{(U_n^{(p)})^{-1}e_i}_p\right)^{-1},\]
so that we can apply Proposition \ref{5.2} to obtain
\[\norm{R_nA|_E}_{2,\infty}\norm{BJ_n}_{p,2}\geq(2\delta)^{-1}=\left(2\cdot 2^{(n-1)(\frac{1}{p'}-\frac{1}{2})}\right)^{-1}.\]
However, this contradicts \eqref{factor-estimate}.  This proves that $I_{q,\infty}\widehat{U}\notin[\mathcal{G}_{\ell_2}](\ell_p,\ell_\infty)$.

Next, suppose towards a contradiction that $\widehat{J}\widehat{U}P\in[\mathcal{G}_{\ell_2}](X,\widehat{Y})$.  Then we can find operators $(\widehat{A}_n)_{n=1}^\infty\subseteq\mathcal{G}_{\ell_2}(X,Y)$ with $\widehat{A}_n\to\widehat{J}\widehat{U}P$ in norm.  Set $\widetilde{Y}:=\widehat{J}\ell_q\subseteq\widehat{Y}$, and let $\widetilde{J}:\ell_q\to\widetilde{Y}$ be the isomorphism induced by $\widehat{J}$, i.e. $\widetilde{J}x=\widehat{J}x$ for all $x\in\ell_q$.  Now let $Q:\ell_p\to X$ be an embedding satisfying $PQ=I_{p,p}$.  Via injectivity of $\ell_\infty$ we may extend $R:=I_{q,\infty}\widetilde{J}^{-1}$ to $\widetilde{R}:\widehat{Y}\to\ell_\infty$.  Then
\[\widetilde{R}\widehat{A}_nQ\to\widetilde{R}\widehat{J}\widehat{U}PQ=I_{q,\infty}\widetilde{J}^{-1}\widetilde{J}\widehat{U}PQ=I_{q,\infty}\widehat{U},\]
contradicting the fact that $I_{q,\infty}\widehat{U}\notin[\mathcal{G}_{\ell_2}](\ell_p,\ell_\infty)$.  It follows that $\widehat{J}\widehat{U}P\notin[\mathcal{G}_{\ell_2}](X,\widehat{Y})$.

Notice that $I_{0,\infty}U=I_{q,\infty}\widehat{U}$ so that $I_{0,\infty}U\notin[\mathcal{G}_{\ell_2}](\ell_p,\ell_\infty)$.  This way, we can run a similar argument as above, by supposing towards a contradiction that $JUP\in[\mathcal{G}_{\ell_2}](X,Y)$.  Then we can find operators $A_n\in\mathcal{G}_{\ell_2}(X,Y)$ such that $A_n\to JUP$.  Let $S:Jc_0\to c_0$ be an isomorphism satisfying $SJ=I_{0,0}$.  Via injectivity we may extend $M:=I_{0,\infty}S$ to $\widetilde{M}:Y\to\ell_\infty$.  Then
\[\widetilde{M}A_nQ\to\widetilde{M}JUPQ=I_{0,\infty}SJUPQ=I_{0,\infty}U,\]
contradicting the fact that $I_{0,\infty}U\notin[\mathcal{G}_{\ell_2}](\ell_p,\ell_\infty)$.  It follows that $JUP\notin[\mathcal{G}_{\ell_2}](X,Y)$.

Now let's suppose $X$ is reflexive.  Denote by $K:Y\to Y^{**}$ and $\widehat{K}:\widehat{Y}\to\widehat{Y}^{**}$ the canonical embeddings.  By Proposition \ref{predual} there exist operators $V^{(p)}\in\mathcal{L}(\ell_p,\ell_{p'})$ and $P_*\in(\ell_{p'},X^*)$ such that $V^{(p)*}=U^{(p)}$ and $(P_*)^*=P$.  Let us define
\[V:=P_*V^{(p)}I_{1,p}J^*\;\;\;\text{ and }\;\;\;\widehat{V}:=P_*V^{(p)}I_{q',p}\widehat{J}^*.\]
Notice that we have $J^{**}|_{c_0}=KJ$, and hence
\[V^*=J^{**}I_{p',\infty}U^{(p)}P=J^{**}|_{c_0}I_{p',0}U^{(p)}P=KJUP.\]
Similarly,
\[\widehat{V}^*=\widehat{J}^{**}I_{p',q}U^{(p)}P=\widehat{J}^{**}\widehat{U}P=\widehat{K}\widehat{J}\widehat{U}P.\]

It was shown in \cite[Theorem 4]{Pl04} that an operator $T$ is class $\mathcal{FSS}$ (resp. $\mathcal{SSCS}$) if and only if $T^*$ is $\mathcal{SSCS}$ (resp. $\mathcal{FSS}$).  In particular, this means $V$ and $\widehat{V}$ are both $\mathcal{SSCS}$.  Also, since $U^*$ and $V$ both factor through $I_{1,p}$, they are each class $\mathcal{FSS}$, and furthermore $U$ is class $\mathcal{SSCS}$.  Now, $\widehat{V}$ factors through a predual $\widehat{U}_*$ of $\widehat{U}$, with $\widehat{U}_*:\ell_{q'}\to\ell_{p'}$ and hence $\widehat{V}$ both $\mathcal{SSCS}$.  Recall that an operator is said to be {\it B-convex} just in case it fails to contain uniformly isomorphic copies of $\ell_1^n$, $n\in\mathbb{Z}^+$.  This is equivalent to having type $r$ for some $r>1$ (cf., e.g., \cite[Remark 2.7]{Pis82}).  Of course, $\ell_p$ has type $p>1$ (cf., e.g., \cite[Theorem 6.2.14]{AK06}) and hence is $B$-convex.  It was also shown in \cite[Theorem 3]{Pl04} that if $X$ is a $B$-convex Banach space, $Y$ is an arbitrary Banach space, and $T:X\to Y$ is $\mathcal{FSS}$, then $T^*$ is $\mathcal{FSS}$.  In particular, $\widehat{U}^*=\widehat{U}_*$ is $\mathcal{FSS}$, and it follows therefore that $\widehat{U}$ is $\mathcal{SSCS}$ and $\widehat{V}$ is $\mathcal{FSS}$.

It remains to show that neither  $V$ nor $\widehat{V}$ are class $[\mathcal{G}_{\ell_2}]$.  First, notice that we have already proved the first part of (i), so that, since $KJ: c_0\to Y^{**}$ is a bounded linear embedding with $Y^{**}$ containing a copy of $c_0$, therefore $KJUP\notin[\mathcal{G}_{\ell_2}](X,Y^{**})$.  Let us assume towards a contradiction that $V\in[\mathcal{G}_{\ell_2}](Y^*,X^*)$.  Then for each $\gamma>0$, there exists $\widetilde{V}\in\mathcal{G}_{\ell_2}(Y^*,X^*)$ such that $\norm{V-\widetilde{V}}<\gamma$.  This gives us $(\widetilde{V})^*\in\mathcal{G}_{\ell_2}(X,Y^{**})$, with
\[\norm{V^*-(\widetilde{V})^*}=\norm{(V-\widetilde{V})^*}=\norm{V-\widetilde{V}}<\gamma.\]
It follows that $KJUP=V^*\in[\mathcal{G}_{\ell_2}](X,Y^{**})$, which contradicts our conclusion from above.  A nearly identical argument shows that $\widehat{V}\notin[\mathcal{G}_{\ell_2}](\widehat{Y}^*,X^*)$.\end{proof}

\begin{corollary}\label{general-diagram}Let $1<p<2<p'\leq q\leq\infty$.  Suppose $X$ is a (real or complex) Banach space containing a complemented copy of $\ell_p$, and that $Y$ is a (real or complex) Banach space containing a copy of either $\ell_q$ if $q\neq\infty$ or $c_0$ if $q=\infty$ (not necessarily complemented).  Then the closed subideals of $\mathcal{L}(X,Y)$ are not linearly ordered.  In particular, we have the following diagram.

\noindent
{\footnotesize
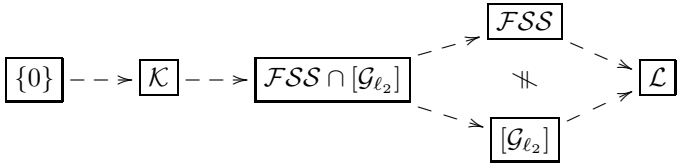
\xymatrix@R=0pt@C=23pt{    % R conrols hight, C - width
 & & & \boxed{\mathcal{FSS}} \ar@{-->}[dr] & \\
\boxed{\{0\}}\ar@{-->}[r] & \boxed{\mathcal{K}} \ar@{-->}[r] &
  \boxed{\mathcal{FSS}\cap[\mathcal{G}_{\ell_2}]} \ar@{-->}[ur] \ar@{-->}[dr] & \bneq &
   \boxed{\mathcal{L}}\\
 & & & \boxed{[\mathcal{G}_{\ell_2}]} \ar@{-->}[ur] & 
}}

\noindent Here, the dashed arrows ($--\!\!\!>$) all represent proper inclusions, and ``$\neq$'' represents incomparable subsets, i.e. that neither set is a subset of the other.

If furthermore $X$ is reflexive, then the closed subideals of $\mathcal{L}(Y^*,X^*)$ are not linearly ordered, and we have the following diagram.

\noindent
{\footnotesize
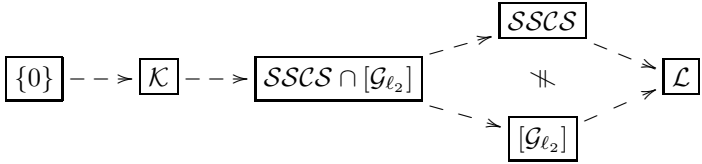
\xymatrix@R=0pt@C=23pt{    % R conrols hight, C - width
 & & & \boxed{\mathcal{SSCS}} \ar@{-->}[dr] & \\
\boxed{\{0\}}\ar@{-->}[r] & \boxed{\mathcal{K}} \ar@{-->}[r] &
  \boxed{\mathcal{SSCS}\cap[\mathcal{G}_{\ell_2}]} \ar@{-->}[ur] \ar@{-->}[dr] & \bneq &
   \boxed{\mathcal{L}}\\
 & & & \boxed{[\mathcal{G}_{\ell_2}]} \ar@{-->}[ur] & 
}}\end{corollary}

\begin{proof}Let $\pi:X\to W$ be a projection onto a subspace $W$ which is isomorphic to $\ell_p$, and let $A:W\to\ell_p$ be an isomorphism.  Consider the case where $q\neq\infty$, and let $\widehat{J}:\ell_q\to Y$ be an embedding.  Then $\widehat{J}I_{p,q}A\pi$ is class $\mathcal{FSS}\cap\mathcal{G}_{\ell_2}$ but not $\mathcal{K}$.  Next, consider the case where $q=\infty$, and let $J:c_0\to Y$ be an embedding.  Then $JI_{p,0}A\pi$ is $\mathcal{FSS}\cap\mathcal{G}_{\ell_2}$ but not $\mathcal{K}$.  Since $\mathcal{FSS}$ is in full duality with $\mathcal{SSCS}$ and each of $\mathcal{K}$ and $\mathcal{G}_{\ell_2}$ is in full duality with itself, this gives us the first two arrows in each diagram.

In case $q\neq\infty$, let $\widehat{T}=D_q^{-1}I_{2,p,q}D_p$ be a $(p,q)$-Pe\l czy\'{n}ski operator.  Then $\widehat{J}\widehat{T}A\pi$ is class $\mathcal{G}_{\ell_2}$ but not $\mathcal{FSS}$ since it uniformly fixes copies of $\ell_2^n$ for all $n\in\mathbb{Z}^+$.  Due to the duality between $\mathcal{FSS}$ and $\mathcal{SSCS}$, its dual $(\widehat{J}\widehat{T}A\pi)^*$ is $\ell_2$-factorable but not $\mathcal{SSCS}$.  
 
Similarly, in case $q=\infty$, let $T=\theta I_{2,p,0}D_p$ be a $(p,0)$-left Pe\l czy\'{n}ski operator.  Then for the same reasons, $JTA\pi$ is class $\mathcal{G}_{\ell_2}$ but not $\mathcal{FSS}$, and its dual $(JTA\pi)^*$ is class $\mathcal{G}_{\ell_2}$ but not $\mathcal{SSCS}$.  Applying Theorem \ref{Hadamard} now completes the proof.\end{proof}

The second main Theorem follows straightforwardly from the above Corollary.

\begin{proof}[Proof of Theorem \ref{main2}]Let $X$ be a (real or complex) Banach space containing a complemented copy of $\ell_p$ and a copy of either $\ell_q$ or $c_0$.  We decompose $X=\ell_p\oplus Y$ for some subspace $Y$ of $X$, so that by Proposition \ref{copy-of-Z}, $Y$ contains a copy of either $\ell_q$ or $c_0$.  Since $p\in(1,2)$, the space $\ell_p$ is reflexive, and so by Corollary \ref{general-diagram} we can find incomparable closed subideals in $\mathcal{L}(\ell_p,Y)$ and $\mathcal{L}(Y^*,\ell_p^*)$.   By Proposition \ref{order-isomorphism}, this means $\mathcal{L}(X)$ and $\mathcal{L}(X^*)$ admit incomparable closed ideals.\end{proof}

\section{Closed ideals in $\mathcal{L}(\ell_p\oplus c_0)$ and $\mathcal{L}(\ell_1\oplus\ell_q)$, $1<p<2<q<\infty$}

In this section we will study the special cases of $\mathcal{L}(\ell_p\oplus c_0)$ and $\mathcal{L}(\ell_1\oplus\ell_q)$ for $1<p<2<q<\infty$.  We shall begin by summarizing what is currently known about the closed ideals in these algebras.  It was proved in \cite[Theorem 5.3.2]{Pi78} that $\mathcal{L}(\ell_p\oplus\ell_q)$, $1\leq p<q<\infty$, has exactly two maximal ideals, and that the lattice of further closed ideals is order-isomorphic to the lattice of closed subideals in $\mathcal{L}(\ell_p,\ell_q)$.  According to Proposition \ref{order-isomorphism}, we also have an injective and order-preserving relationship between the closed subideals in $\mathcal{L}(\ell_p,c_0)$, $1\leq p<\infty$, and the closed ideals in $\mathcal{L}(\ell_p\oplus c_0)$.  In \cite[Proposition 3.1]{SSTT07}, the authors showed that any ideal of $\mathcal{L}(\ell_1,\ell_q)$ containing a noncompact operator must also contain $I_{1,q}$.  By replacing ``$\ell_q$'' with ``$c_0$'' as needed in their proof, we obtain the analogous conclusion that any ideal of $\mathcal{L}(\ell_p,c_0)$ containing a noncompact operator must also contain $I_{p,0}$.

It was shown in \cite[Proposition 4]{Pl04} that $\mathcal{FSS}(\ell_p,c_0)$ is a proper closed subideal of $\mathcal{L}(\ell_p,c_0)$ for $1<p<\infty$.  In contrast, in \cite[Remark 5]{Pl04} the author observed that $\mathcal{FSS}(\ell_1,\ell_q)=\mathcal{L}(\ell_1,\ell_q)$.  However, we can consider superstrictly cosingular operators in place of $\mathcal{FSS}$.  Indeed, by duality we have that $\mathcal{SSCS}(\ell_1,\ell_q)$ is a proper closed subideal in $\mathcal{L}(\ell_1,\ell_q)$.

The following diagram captures the facts we have summarized so far regarding the lattice of closed subideals in $\mathcal{L}(\ell_p,c_0)$, $1<p<\infty$.

{\footnotesize \xymatrix@R=0pt@C=23pt{ \boxed{\{0\}} \ar@{=>}[r] & \boxed{\mathcal K} \ar@{=>}[r] &    \boxed{[\mathcal{G}_{I_{p,0}}]} \ar@{.>}[r] &   \boxed{\mathcal{FSS}} \ar@{-->}[r] & \boxed{\mathcal{L}} }}  % R conrols hight, C - width
\noindent The notation comes from \cite{SSTT07}.  As in that paper, the various types of arrows represent inclusions.  A solid single-bar arrow ($\to$) is an immediate successor (i.e., no ideals sitting in between), while a double solid arrow ($\Rightarrow$) denotes a unique immediate successor.  A hyphenated arrow ($--\!\!\!>$) represents a proper inclusion, and a dotted arrow ($\cdots\!\!\!>$) is an inclusion which we do not know whether it is proper.

Let us also give a diagram of the facts so far regarding the closed subideal structure of $\mathcal{L}(\ell_1,\ell_q)$, $1<q<\infty$.

{\footnotesize \xymatrix@R=0pt@C=23pt{ \boxed{\{0\}} \ar@{=>}[r] & \boxed{\mathcal K} \ar@{=>}[r] &    \boxed{[\mathcal{G}_{I_{1,q}}]} \ar@{.>}[r] &   \boxed{\mathcal{SSCS}} \ar@{-->}[r] & \boxed{\mathcal{L}} }} 

Fix $1<p<2$.  Let $I_{\bv,\infty}$ be any of the operators from Theorem \ref{main1-utility} (with $q=\infty$), and let $T$ be a $(p,0)$-left Pe\l czy\'{n}ski operator.  We will show in this section that the following diagram represents part of the closed subideal structure of $\mathcal{L}(\ell_p,c_0)$.

\noindent
{\footnotesize
\xymatrix@R=0pt@C=23pt{    % R conrols hight, C - width
\boxed{\{0\}} \ar@{=>}[r] & \boxed{\mathcal K} \ar@{=>}[r] & 
  \boxed{[\mathcal{G}_{I_{p,0}}]} \ar@{-->}[r] &
  \boxed{[\mathcal{G}_{I_{\bv,\infty}}]}% \ar@{-->}[r]
}}

\noindent
{\footnotesize
\xymatrix@R=0pt@C=23pt{    % R conrols hight, C - width
 \;\;\;\;\;\;\;\;\;\;\;\;\;\;\;\;\;\;\;\;\;\;\;\;\;\;\;\;\;\;\;\;\;\;\;\;\;\;\;\;\; & & & \boxed{\mathcal{FSS}} \ar@{=>}[dr] & & \\
 \;\;\;\;\;\;\;\;\;\;\;\;\;\;\;\;\;\;\;\;\;\;\;\;\;\;\;\;\;\;\;\;\;\;\;\;\;\;\;\;\; & \ar@{-->}[r] & 
  \boxed{\mathcal{FSS}\cap[\mathcal{G}_{\ell_2}]} \ar@{-->}[ur] \ar[dr] & \bneq &
  \boxed{[\mathcal{FSS}+\mathcal{G}_T]} \ar@{.>}[r] & \boxed{\mathcal{L}}
  \\
 \;\;\;\;\;\;\;\;\;\;\;\;\;\;\;\;\;\;\;\;\;\;\;\;\;\;\;\;\;\;\;\;\;\;\;\;\;\;\;\;\; & & & \boxed{[\mathcal{G}_T]} \ar@{-->}[ur] & &
}}

\noindent Here, the ``not equal to'' symbol ($\neq$) means that neither subideal is a subspace of the other.  Note again that the new diagram for $\mathcal{L}(\ell_p,c_0)$ is only proved for $1<p<2$.

Using duality, we will also prove the following structure for the closed subideals of $\mathcal{L}(\ell_1,\ell_q)$ for all $2<q<\infty$.

\noindent
{\footnotesize
\xymatrix@R=0pt@C=23pt{    % R conrols hight, C - width
& & & & & \boxed{\mathcal{SSCS}} \ar@{-->}[dr] & \\
\boxed{\{0\}} \ar@{=>}[r] & \boxed{\mathcal K} \ar@{=>}[r] & 
  \boxed{[\mathcal{G}_{I_{1,q}}]} \ar@{-->}[r] & \boxed{[\mathcal{G}_{I_{\bv,\infty}}^*]}\ar@{-->}[r] &
  \boxed{\mathcal{SSCS}\cap[\mathcal{G}_{\ell_2}]} \ar@{-->}[ur] \ar@{-->}[dr] & \bneq & \boxed{\mathcal{L}}\\
& & & & & \boxed{[\mathcal{G}_{\ell_2}]} \ar@{-->}[ur] &
}}

\noindent Again, note that this diagram only holds for $2<q<\infty$.

We begin our proof by giving some basic norm estimates, which we will need momentarily.  Recall that if $X$ is a Banach space with a basis $(x_n)_{n=1}^\infty$, and $k\in\mathbb{Z}^+$, then the {\bf $\boldsymbol{k}$th partial sum projection} with respect to $(x_n)_{n=1}^\infty$ is the continuous linear operator $S_k\in\mathcal{L}(X)$ defined by
\[S_k\sum_{n=1}^\infty a_nx_n=\sum_{n=1}^k a_nx_n\;\;\;\text{ for all }\;\;\;x=\sum_{n=1}^\infty a_nx_n\in X.\]
Where convenient, we shall define $S_0=0$.

\begin{proposition}\label{truncate-1}Let $X$ be a (real or complex) Banach space with a basis and corresponding and partial sum projections $(S_k)_{k=1}^\infty$, and let $E$ be finite-dimensional subspace of $X$.  For every $\delta>0$ there exists $N\in\mathbb{Z}^+$ such that $\norm{e-S_ke}\leq\delta\norm{e}$ for all $e\in E$ and $k\geq N$.\end{proposition}

\begin{proof}Let $\{e_1,\cdots,e_n\}$, $n=\dim(E)$, be a normalized basis for $E$, and let $K>0$ be such that it is $K$-equivalent to the canonical basis of $\ell_1^n$.  Notice that for each $i=1,\cdots,n$, we can find $k_i\in\mathbb{Z}^+$ such that $\norm{e_i-S_ke_i}\leq\delta/K$ for all $k\geq k_i$.  Let $N:=\max\{k_1,\cdots,k_n\}$.  Then for any $e=\sum_{i=1}^na_ie_i\in E$ and $k\geq N$ we have
\[\norm{e-S_ke}=\norm{\sum_{i=1}^na_i(e_i-S_ke_i)}\leq\frac{\delta}{K}\sum_{i=1}^n|a_i|\leq\delta\norm{\sum_{i=1}^na_ie_i}=\delta\norm{e}.\]\end{proof}

\begin{proposition}\label{truncate-2}Let $X$ be a (real or complex) Banach space with a basis and corresponding partial sum projections $(S_k)_{k=1}^\infty\subseteq\mathcal{L}(X)$, and let $E$ be a $n$-dimensional subspace, $n\in\mathbb{Z}^+$, of $X$.  Let $T\in\mathcal{L}(X,Y)$ for some Banach space $Y$, and such that $T|_E$ is bounded below by $\epsilon>0$, i.e. $\norm{Te}\geq\epsilon\norm{e}$ for all $e\in E$.  Then for every $\delta\in(0,\epsilon)$, there exists $N\in\mathbb{Z}^+$ such that $S_NE$ is $n$-dimensional and $T|_{(S_NE)}$ is bounded below by $\delta$, i.e. $\norm{TS_Ne}\geq\delta\norm{S_Ne}$ for all $e\in E$.\end{proposition}

\begin{proof}Since $\delta\in(0,\epsilon)$, we can find $\gamma\in(0,1)$ such that
\[\delta\leq\frac{\epsilon}{1+\gamma}-\frac{\norm{T}\gamma}{1-\gamma}.\]
By Proposition \ref{truncate-1}, c $N\in\mathbb{Z}^+$ such that $\norm{e-S_Ne}\leq\gamma\norm{e}$ for all $e\in E$.  Thus,
\[\norm{e}=\frac{1}{1-\gamma}(\norm{e}-\gamma\norm{e})\leq\frac{1}{1-\gamma}\left(\norm{S_Ne}+\norm{e-S_Ne}-\gamma\norm{e}\right)\leq\frac{1}{1-\gamma}\norm{S_Ne}.\]
This shows that $\dim(S_NE)=\dim(E)$.  Together with
\[\norm{S_Ne}\leq\norm{e}+\norm{e-S_Ne}\leq(1+\gamma)\norm{e},\]
it also gives us
\begin{multline*}\frac{\epsilon}{1+\gamma}\norm{S_Ne}\leq\epsilon\norm{e}\leq\norm{Te}\leq\norm{TS_Ne}+\norm{T}\norm{e-S_Ne}\\\leq\norm{TS_Ne}+\norm{T}\gamma\norm{e}\leq\norm{TS_Ne}+\frac{\norm{T}\gamma}{1-\gamma}\norm{S_Ne},\end{multline*}
and finally
\[\delta\norm{S_Ne}\leq\left(\frac{\epsilon}{1+\gamma}-\frac{\norm{T}\gamma}{1-\gamma}\right)\norm{S_Ne}\leq\norm{TS_Ne}.\]\end{proof}

We will also need to recall the following obvious consequence of the Second Isomorphism Theorem for modules.

\begin{proposition}\label{intersection-dimension}Let $E$ be a finite-dimensional subspace and let $Z$ be an infinite-dimensional subspace of a vector space $X$.  Then
\[\dim(E)-\dim(X/Z)\leq\dim(E\cap Z).\]\end{proposition}

\begin{proof}By the Second Isomorphism Theorem for modules we have
\[E/(E\cap Z)\cong(E+Z)/Z\]
so that
\[\dim(E)-\dim(E\cap Z)=\dim(E/(E\cap Z))=\dim((E+Z)/Z)\leq\dim(X/Z)\]
and hence
\[\dim(E)-\dim(X/Z)\leq\dim(E\cap Z).\]\end{proof}

Now let us prove the following Theorem, which is very closely analogous to \cite[Theorem 4.11]{SSTT07}.  Our methods are almost identical to theirs, and so we will not deviate too far from their proof.  Let us use the following notation in our proof.  If $X$ is a Banach space with a basis $(x_n)$, then for any $x\in X$ we write
\[\text{supp}(x)=\left\{n\in\mathbb{Z}^+:a_n\neq 0,x=\sum_{n=1}^\infty a_nx_n\right\},\]
and if $E$ is a subset of $X$ then we define
\[\text{supp}(E)=\bigcup_{e\in E}\text{supp}(e).\]

\begin{theorem}\label{LPDO}Let $1<p<\infty$, and let $T\in\mathcal{L}(\ell_p,c_0)$ be any $(p,0)$-left Pe\l czy\'{n}ski decomposition operator.  If $R\in\mathcal{L}(\ell_p,c_0)\setminus\mathcal{FSS}(\ell_p,c_0)$ then $T$ factors through $R$.\end{theorem}

\begin{proof}Let $(k_n)_{n=1}^\infty$, $\theta=\bigoplus_{n=1}^\infty\theta_n$, and $I_{2,p,0}$ be as in the definition of the $(p,0)$-left Pe\l czy\'{n}ski decomposition operator $T=\theta I_{2,p,0}D_p$.  It is enough to show that $\theta I_{2,p,0}$ factors through $R$.  To do it, we will closely follow the proof of \cite[Theorem 4.11]{SSTT07}, only making some crucial modifications along the way.

Since $R$ is not $\mathcal{FSS}$, there exists $\epsilon>0$ and a sequence $(E_n)_{n=0}^\infty$ of subspaces of $\ell_p$ such that $\dim(E_n)=n$ for all $n\in\mathbb{Z}^+$, and such that $\norm{Rx}\geq\epsilon\norm{x}$ for all $x\in\bigcup_{n=1}^\infty E_n$.  Due to Proposition \ref{truncate-2}, we can assume $\text{supp}(E_n)<\infty$ for all $n\in\mathbb{N}$, adjusting $\epsilon$ if necessary.  Let $(F_n)_{n=1}^\infty$ be defined by setting $F_n:=RE_n$ for each $n\in\mathbb{N}$, and let $(S_k)_{k=0}^\infty$ be the canonical partial sum projections lying in $\mathcal{L}(c_0)$.  Choose $\gamma\in(0,1/2)$ to satisfy $(1+\gamma)/(1-\gamma)\leq 2$.

Let's inductively construct sequences $(\widehat{E}_n)_{n=0}^\infty$ and $(\widehat{F}_n)_{n=0}^\infty$, and a strictly increasing sequence $(m_n)_{n=0}^\infty\subseteq\mathbb{N}$, such that the following conditions are satisfied for all $n\in\mathbb{Z}^+$.
\begin{itemize}\item[(i)]  $m_{n-1}<\text{supp}(\widehat{E}_n)$;
\item[(ii)]  $m_{n-1}<\text{supp}(\widehat{F}_n)$;
\item[(iii)]  $\text{supp}(\widehat{E}_n)\leq m_n$;
\item[(iv)]  $\widehat{F}_n=R\widehat{E}_n$;
\item[(v)]  $\norm{Rx}\geq\epsilon\norm{x}$ for all $x\in\widehat{E}_n$;
\item[(vi)]  $\norm{y-S_{m_n}y}\leq\gamma 2^{-n}\norm{y}$ for all $y\in\widehat{F}_n$; and
\item[(vii)]  $\dim(\widehat{E}_n)=\dim(\widehat{F}_n)=n$.\end{itemize}
First, set $m_0=0$, $\widehat{E}_0=\{0\}$, and $\widehat{F}_0=\{0\}$, and suppose that we have constructed $\widehat{E}_i$, $\widehat{F}_i$, and $m_i$, for all $i<n$ and some $n\in\mathbb{Z}^+$.  Let $G$ and $G'$ be the subspaces of $\ell_p$ and $c_0$, respectively, consisting of all the vectors whose first $m_{n-1}$ coordinates are zero.  Put $k:=2m_{n-1}+n$.  Due to Proposition \ref{intersection-dimension}, we now have
\begin{multline*}m_{n-1}+n=(2m_{n-1}+n)-m_{n-1}=\dim(F_k)-\dim(c_0/G')\\\leq\dim(F_k\cap G')=\dim(R|_{E_k}^{-1}(F_k\cap G')).\end{multline*}
Again due to Proposition \ref{intersection-dimension}, we get
\begin{multline*}n=(m_{n-1}+n)-m_{n-1}\leq\dim(R|_{E_k}^{-1}(F_k\cap G'))-\dim(\ell_p/G)\\\leq\dim(R|_{E_k}^{-1}(F_k\cap G')\cap G)\end{multline*}
Let $\widehat{E}_n$ be an $n$-dimensional subspace of $R|_{E_k}^{-1}(F_k\cap G')\cap G$, and set $\widehat{F}_n=R\widehat{E}_n$.  Then (i), (ii), (iv), and (vii) are all satisfied for this $n$.   Notice that $\widehat{E}_n\subseteq E_k$, so that (v) is also satisfied.  By Proposition \ref{truncate-1}, we can find $N\in\mathbb{Z}^+$ such that $\norm{y-S_ky}\leq\gamma 2^{-n}\norm{y}$ for all $y\in F_n$ and $k\geq N$.  If we pick $m_n=\max\{\max\text{supp}(\widehat{E}_n),N\}$, this satisfies (iii) and (vi), and the construction is complete.

For convenience, let us relabel $E_n=\widehat{E}_n$ and $F_n=\widehat{F}_n$ for all $n\in\mathbb{Z}^+$.  Recall that for each $k\in\mathbb{Z}^+$ there exists $n_k\in\mathbb{Z}^+$ such that every $n_k$-dimensional subspace of $\ell_p$ contains an $k$-dimensional subspace which is 2-isomorphic to $\ell_2^k$ (cf., e.g., \cite[Theorem 12.3.3]{AK06}).  Thus, by passing to subspaces of a suitable subsequence, we can assume that each $E_n$ is 2-isomorphic to $\ell_2^n$.  Then pass to a matching subsequence of $(m_n)$, and relabel each $F_n=RE_n$, so that properties (i)-(vii) above are preserved.  In addition to these properties, for each $n\in\mathbb{Z}^+$, there now exists an isomorphism $U_n:\ell_2^n\to E_n$ such that $\norm{U_n}\leq 2$ and $\norm{U_n^{-1}}\leq 2$.

We claim that any normalized sequence $(y_n)_{n=1}^\infty\subseteq c_0$ such that $y_n\in F_n$ for each $n\in\mathbb{Z}^+$ is basic and 2-equivalent to the canonical basis $(f_n)_{n=1}^\infty$ of $c_0$; in particular this means the $F_n$'s are all linearly independent.  Indeed, due to property (ii) above, the sequence $(S_{m_n}y_n)_{n=1}^\infty$ is a block sequence of $(f_n)_{n=1}^\infty$.  Also, by definition of the $c_0$-norm together with property (vi), we have
\[1=\norm{y_n}=\max\{\norm{S_{m_n}y_n},\norm{y_n-S_{m_n}y_n}\}\leq\max\{\norm{S_{m_n}y_n},\gamma 2^{-n}\}.\]
Since $\gamma 2^{-n}<1$ and $\norm{S_{m_n}y_n}\leq\norm{y_n}=1$, this means $\norm{S_{m_n}y_n}=1$.  Thus, $(S_{m_n}y_n)_{n=1}^\infty$ is a normalized block basis of $(f_n)_{n=1}^\infty$, so that it is 1-equivalent to $(f_n)$ (cf., e.g., \cite[Lemma 2.1.1]{AK06}).  On the other hand, notice that
\[2\sum_{n=1}^\infty\frac{\norm{y_n-S_{m_n}y_n}}{\norm{S_{m_n}y_n}}\leq 2\sum_{n=1}^\infty 2^{-n}\gamma=2\gamma<1\]
so that by the Principle of Small Perturbations (cf., e.g., \cite[Theorem 1.3.9]{AK06}), $(y_n)_{n=1}^\infty$ is $(1+\gamma)/(1-\gamma)$-equivalent to $(S_{m_n}y_n)_{n=1}^\infty$ and hence, due to $(1+\gamma)/(1-\gamma)\leq 2$, it is 2-equivalent to $(f_n)_{n=1}^\infty$.

For each $n\in\mathbb{Z}^+$, define $R_n:E_n\to F_n$ by the rule $R_nx=Rx$ for all $x\in E_n$.  Then each $R_n$ is an invertible operator satisfying $\norm{R_n}\leq\norm{R}$ and $\norm{R_n^{-1}}\leq 1/\epsilon$.  For each $n\in\mathbb{Z}^+$, let us also define an operator
\[J_n=\theta_nU_n^{-1}R_n^{-1}:F_n\to\ell_\infty^{k_n}.\]
Notice that this means $\norm{J_n}\leq 4/\epsilon$ for all $n\in\mathbb{Z}^+$.  Let us also, for each $n\in\mathbb{Z}^+$, denote by
\[Q_n:\ell_\infty^{k_n}\to\left(\bigoplus_{n=1}^\infty\ell_\infty^{k_n}\right)_{c_0}\]
the canonical norm-1 embedding.  Due to the linear independence of the $F_n$'s, we can now define a linear map
\[J:\text{span}\bigcup_{n=1}^\infty F_n\to\left(\bigoplus_{n=1}^\infty\ell_\infty^{k_n}\right)_{c_0}\]
by the rule $Jy=Q_nJ_ny$ for all $y\in F_n$ and $n\in\mathbb{Z}^+$.  Let us show that $J$ is bounded.  For any nonzero $y\in\text{span}\bigcup_{n=1}^\infty F_n$, we can write $y=\sum_{k=1}^jy_k$ for some $j\in\mathbb{Z}^+$, where $(F_{i_k})_{k=1}^j$ is a subsequence and $y_k\in F_{i_k}\setminus\{0\}$ for each $k=1,\cdots,j$.  Since every normalized basic sequence formed by single elements in each $F_n$ is 2-equivalent to $(f_n)_{n=1}^\infty$, this gives us
\[\norm{Jy}=\norm{J\sum_{k=1}^jy_k}=\sup_{1\leq k\leq j}\norm{Q_{i_k}J_{i_k}y_k}\leq\frac{4}{\epsilon}\sup_{1\leq k\leq j}\norm{y_k}\leq\frac{8}{\epsilon}\norm{\sum_{k=1}^j\norm{y_k}\frac{y_k}{\norm{y_k}}}=\frac{8}{\epsilon}\norm{y}.\]
Thus, $J$ extends to
\[\widetilde{J}:c_0\to\left(\bigoplus_{n=1}^\infty\ell_\infty^{k_n}\right)_{c_0}\]
via the separable injectivity of $\left(\bigoplus_{n=1}^\infty\ell_\infty^{k_n}\right)_{c_0}\cong c_0$ (cf., e.g., \cite[VII, p72]{Di84}).

Next, define an operator
\[U:\left(\bigoplus_{n=1}^\infty\ell_2^n\right)_{\ell_p}\to\ell_p\]
by the rule
\[U\bigoplus_{n=1}^\infty x_n=\sum_{n=1}^\infty U_nx_n.\]
Actually, it is not yet clear that it is possible to define $U$, except of finitely-supported $\oplus_{\ell_p}$-sums.  Let us show that when restricted to finite support, $U$ is bounded, and hence well-defined on the whole space via continuous extension.  Since the $E_n$'s are disjointly supported in $\ell_p$, and $\norm{U_n}\leq 2$ for all $n\in\mathbb{Z}^+$, we have
\[\norm{U\bigoplus_{n=1}^\infty x_n}=\norm{\sum_{n=1}^\infty U_nx_n}=\left(\sum_{n=1}^\infty\norm{U_nx_n}^p\right)^{1/p}\leq 2\left(\sum_{n=1}^\infty\norm{x_n}^p\right)^{1/p}=2\norm{\bigoplus_{n=1}^\infty x_n}.\]
It follows, as claimed, that $U$ is bounded on finite $\oplus_{\ell_p}$-sums, and hence is a well-defined bounded operator on the whole space.

Observe that we now have $\theta I_{2,p,0}=\widetilde{J}RU$, which completes the proof.\end{proof}

Let us now show how to deduce the new diagrams above for $\mathcal{L}(\ell_p,c_0)$ and $\mathcal{L}(\ell_1,\ell_q)$, $1<p<2<q<\infty$.  We don't need to prove the first two arrows in each diagram, since they are already known.  Each third and fourth arrows follow from Theorem \ref{main1-utility} together with the fact that the operators $I_{\bv,\infty}$ are all $\ell_2$-factorable.  The rest of the diagram for $\mathcal{L}(\ell_1,\ell_q)$ follows from Corollary \ref{general-diagram}.  Now let $T$ be a $(p,0)$-left Pe\l czy\'{n}ski operator, which we have already observed is not $\mathcal{FSS}$.  Thus, by Theorem \ref{LPDO} together with the fact that $T$ is $\ell_2$-factorable, $[\mathcal{G}_T](\ell_p,c_0)$ is an immediate successor to $\mathcal{FSS}\cap[\mathcal{G}_{\ell_2}](\ell_p,c_0)$, and $[\mathcal{FSS}+\mathcal{G}_T](\ell_p,c_0)$ is the only immediate successor to $\mathcal{FSS}(\ell_p,c_0)$.  Theorem \ref{Hadamard} together with the fact that $T$ is $\ell_2$-factorable but not $\mathcal{FSS}$ shows us that $\mathcal{FSS}(\ell_p,c_0)$ and $[\mathcal{G}_T](\ell_p,c_0)$ are incomparable, i.e. neither one is a subset of the other.  It is also clear from these facts that $\mathcal{FSS}\cap[\mathcal{G}_{\ell_2}](\ell_p,c_0)$ is a proper subset of $\mathcal{FSS}(\ell_p,c_0)$ and $[\mathcal{G}_T](\ell_p,c_0)$ is a proper subset of $[\mathcal{FSS}+\mathcal{G}_T](\ell_p,c_0)$.  Thus, the new diagrams are proved.

\end{document}